\documentclass[a4paper,12pt]{amsart}

\usepackage{amssymb, amsmath, amsthm, mathrsfs, braket, xspace}
\usepackage[margin=1.0in]{geometry}

\numberwithin{equation}{section}
\usepackage[colorlinks=true]{hyperref}

\usepackage{mathtools}
\mathtoolsset{showonlyrefs=true} 
\usepackage{enumitem}
\setlist[enumerate]{label=$(\mathrm{\arabic*})$}

\usepackage[all,2cell]{xy}
\objectmargin+{1mm}
\labelmargin+{0.8mm}
\SelectTips{cm}{12}
\usepackage{tikz}
\usepackage{float}
\usepackage{aliascnt}

\newtheorem{thm}{Theorem}[section]

\newaliascnt{cor}{thm}
\newtheorem{cor}[cor]{Corollary}
\aliascntresetthe{cor}

\newaliascnt{lem}{thm}
\newtheorem{lem}[lem]{Lemma}
\aliascntresetthe{lem}

\newaliascnt{prop}{thm}
\newtheorem{prop}[prop]{Proposition}
\aliascntresetthe{prop}

\newaliascnt{conj}{thm}
\newtheorem{conj}[conj]{Conjecture}
\aliascntresetthe{conj}

\theoremstyle{definition}
\newaliascnt{dfn}{thm}

\aliascntresetthe{dfn}

\newaliascnt{rem}{thm}
\newtheorem{rem}[rem]{Remark}
\aliascntresetthe{rem}

\newaliascnt{prob}{thm}

\aliascntresetthe{prob}

\newaliascnt{ex}{thm}

\aliascntresetthe{ex}

\newcommand{\Q}{\mathbb{Q}}

\newcommand{\C}{\mathbb{C}}
\newcommand{\Z}{\mathbb{Z}}

\newcommand{\bF}{\mathbb{F}}

\DeclareMathOperator{\ch}{char}
\DeclareMathOperator{\Coker}{Coker}
\DeclareMathOperator{\Cor}{Cor}

\DeclareMathOperator{\Hom}{Hom}

\DeclareMathOperator{\Ker}{Ker}
\DeclareMathOperator{\res}{res}
\DeclareMathOperator{\Spec}{Spec}
\DeclareMathOperator{\Tr}{Tr}

\DeclareMathOperator{\trdeg}{trdeg}

\newcommand{\cd}{\operatorname{cd}}

\newcommand{\et}{\mathrm{et}}
\newcommand{\Zar}{\mathrm{Zar}}

\newcommand{\Gm}{\mathbb{G}_{m}}

\newcommand{\sep}{\mathrm{sep}}

\title{Divisibility and torsion in higher Chow groups over arithmetic fields}

\author[T. Hiranouchi]{Toshiro Hiranouchi}\address[T. Hiranouchi]{
Department of Basic Sciences, Graduate School of Engineering, 
Kyushu Institute of Technology, 
1-1 Sensui-cho, Tobata-ku, Kitakyushu-shi, 
Fukuoka, 804-8550 JAPAN}
\email{hira@mns.kyutech.ac.jp}

\author[R. Sugiyama]{Rin Sugiyama} \address[R. Sugiyama]{
Department of Mathematics, Physics and Computer Science, 
Japan Women's University, 2-8-1 Mejirodai, Bunkyo-ku, Tokyo, 112-8681 JAPAN
}\email{sugiyamar@fc.jwu.ac.jp}

\begin{document}

\begin{abstract}
Let $X$ be a smooth scheme of dimension $d$ over a field $F$. We study
the abelian-group structure of the higher Chow groups $CH^{d+i}(X,j)$.
For a prime $l$ different from the characteristic of $F$, 
we prove divisibility and torsion-freeness
results when $i\ge$ the $l$-cohomological dimension of $F$. 
If $X$ is smooth proper and geometrically irreducible, we also study the kernel of the push-forward
$CH^{d+i}(X,j)\to CH^i(F,j)$. We apply these results to finite fields,
local fields, and global fields. 
\end{abstract}

\maketitle

\section{Introduction}
Bloch's higher Chow groups $CH^m(X,n)$ provide a
cycle-theoretic realization of motivic cohomology. For $n=0$, they
recover the ordinary Chow groups:
\[
CH^m(X,0)=CH^m(X).
\]
For $n>0$, they contain additional arithmetic information and are
related to algebraic $K$-theory through localization sequences and
the isomorphism
\[
CH^n(\Spec(F),n) =:CH^n(F,n)\simeq K_n^M(F).
\]
The purpose of this paper is to study the
abelian-group structure of such higher Chow groups in 
$m>d$, where $d=\dim X$, with particular emphasis on finiteness, divisibility, and 
torsion-freeness.

More precisely, 
let $X$ be a smooth proper geometrically irreducible scheme over a field $F$ of dimension $d$. 
We are mainly
interested in groups of the form
$CH^{d+i}(X,j)$ for non-negative integers $i,j$. 
Our motivation comes from Akhtar's work on higher Chow groups of
zero-cycles and one-cycles over finite and global fields, which we
recall below.
In the following $\ch(F)$ denotes the characteristic of $F$.

\begin{thm}[{\cite[Theorem~1.2 and Theorem~1.3]{Akh04b}, \cite[Proposition 4.2]{Akh04b}}] 
	\label{thm:Akh}
	Suppose that $F = \bF_q$ is a finite field of characteristic $p$. 

	\begin{enumerate}
		\item We have 
		\[CH^{d+i}(X,i) \simeq \begin{cases}
			0 ,& \mbox{if $i\ge 2$},\\
			\bF_q^\times,& \mbox{if $i=1$}.\\
		\end{cases}
		\]
	\item $CH^{d+i}(X,i+1)$ is torsion for every $i\ge 1$. 
	\item $CH^{d+i}(X,i+1)$ is $p$-divisible for $i\geq 2$. 
	If, in addition, $X$ is a curve, then $CH^{i+1}(X,i+1)$ is uniquely $p$-divisible. 
	\end{enumerate}  
\end{thm}

In \cite[Theorem~1.6]{GKR26}, 
over a finite field $\bF_q$ of $\ch(\bF_q) = p$, 
Gupta--Krishna--Rathore proved that the prime-to-$p$ torsion part 
$CH^{d+i}(X,j)\{p'\}$ is finite for every $i\ge1$ and $j\ge0$.
Their work gives strong control of the torsion subgroup, but it does not
by itself determine the whole abelian-group structure. 
Our main goal is to give a uniform description of
$CH^{d+i}(X,j)$ in terms of the integer $2i-j$. 
Our first result is a general statement over a field of finite
$l$-cohomological dimension.

\begin{thm}[{\autoref{prop:l-div}}]\label{thm:CHdiv-intro}
Let $F$ be a field with $\cd_l(F)=s$, where $l\neq\ch(F)$, and let
$X$ be a smooth scheme over $F$ of dimension $d$.
Assume that $i\ge s$. Then
\[
CH^{d+i}(X,j)= 
\begin{cases}
\text{uniquely $l$-divisible},
    &\mbox{if $2i-j\ge s+2$},\\
\text{$l$-divisible},
    &\mbox{if $2i-j=s+1$},\\
\text{$l$-torsion free},
    &\mbox{if $2i-j=-2d$},\\
\text{uniquely $l$-divisible},
    &\mbox{if $2i-j<-2d$}.
\end{cases}
\]
\end{thm}
In characteristic $\ch(F) = p>0$, the prime-to-$p$ statement is complemented by
a $p$-primary result. If $[F:F^p]\le p^{s-1}$ and $i\ge s$, then
$CH^{d+i}(X,j)$ is uniquely $p$-divisible (\autoref{prop:uniq-p-div}). 
The proof uses the localization spectral
sequence, the theorem of Geisser--Levine, and the
Bloch--Gabber--Kato theorem.

We next apply the general results to \emph{arithmetic fields}, namely, 
finite fields, local fields, and global fields. 
The finite field case gives the cleanest structure theorem. In the statements below,
an expression such as ``$CH^{d+i}(X,j)$ $\simeq$ finite $\oplus$ uniquely divisible'' means that
there exists a non-canonical direct-sum decomposition with a finite
summand and a uniquely divisible summand. The same convention applies
to the analogous expressions used for local and global fields. Precise
statements, including the canonical short exact sequences when
available, are given in the corresponding results in \autoref{sec:application}.

\begin{thm}[{\autoref{cor:fin2}}]
\label{cor:fin2-into}
Let $X$ be a smooth proper and geometrically irreducible scheme over $\bF_q$
of dimension $d$, and let
$p=\ch(\bF_q)$.
Assume that $i\ge1$. Then
\[
CH^{d+i}(X,j)\simeq
\begin{cases}
\Z/(q^i-1)\oplus\text{uniquely divisible},
    &\mbox{if $2i-j=1$},\\
\text{finite}\oplus\text{uniquely divisible},
    &\mbox{if $-2d<2i-j<1$},\\
\text{uniquely divisible},
    &\mbox{otherwise}.
\end{cases}
\]
In the middle range, the finite group has order prime to $p$.
\end{thm}

If $X$ is projective and Parshin's
conjecture holds for $X$, the uniquely divisible parts in the above theorem vanish. 
Thus the groups are finite in the remaining nonzero ranges; see
\autoref{cor:fin}.
This extends Akhtar's results (\autoref{thm:Akh}) from the cases
$j=i$ and $j=i+1$ to all $j$.

For a local field $F$ with residue characteristic $p$, the group structure of $CH^{d+i}(X,j)$  is
more complicated. 
The classical structure of $CH^i(F,i)\simeq K_i^M(F)$ already shows
that the residue prime behaves differently from the other primes.
For local fields we obtain a more refined structure theorem. 

\begin{thm}[{\autoref{cor:CH-loc}}]
Let $F$ be a local field with finite residue field $k$, and put $p=\ch(k)$.
Let $X$ be a smooth proper and geometrically irreducible scheme over $F$
of dimension $d$. Assume that $i\ge2$. Then there exist
non-canonical decompositions of the following forms:
\[
CH^{d+i}(X,j)
\simeq
\begin{cases}
\text{uniquely divisible},
    & \mbox{if $2i-j\ge3$},\\
CH^i(F,j)_{\mathrm{tor}}\oplus\text{uniquely divisible},
    & \mbox{if $2i-j=2$},\\
\text{finite}\oplus\text{uniquely $p'$-divisible},
    & \mbox{if $-2d<2i-j\le1$ and $\ch(F)=0$},\\
\text{finite}\oplus\text{uniquely divisible},
    & \mbox{if $-2d<2i-j\le1$ and $\ch(F)>0$},\\
\text{uniquely divisible},
    & \mbox{if $2i-j\le-2d$},
\end{cases}
\]
Here, 
an abelian group $G$ is called
\emph{uniquely $p'$-divisible} if it is uniquely $l$-divisible for every prime
$l\neq p$.
\end{thm}
The proof uses local Tate duality, the structure of Galois cohomology of
Tate twists over local fields, and weight arguments for the Tate module of
the Albanese variety.

Finally, let $F$ be a global field. 
We recall that a \textbf{global field} is a finite extension of $\Q$ or a 
function field of one variable over a finite field. 
The following results are known. 

\begin{prop}[{\cite[Corollary 7.2]{Akh04a}}]\label{prop:Akh04_7.2}
    Assume that $F$ is a global field. 
    \begin{enumerate}
        \item If $\ch(F) = 0$, then we have 
    $CH^{d+i}(X,i) \simeq  
    \begin{cases} 
    torsion, & i = 2,\\
    2\text{-}torsion & i\ge 3.
    \end{cases}$
    \item If $\ch(F) >0$, then 
    $CH^{d+i}(X,i) \simeq 
    \begin{cases} 
    torsion, & i = 2,\\
    0 & i\ge 3.
    \end{cases}$
    \end{enumerate}
\end{prop}

Under the assumption that 
$\ch(F)>0$ or $F$ is a totally imaginary number field, 
we obtain the following structure results.

\begin{thm}[\autoref{cor:global}]
Let $F$ be a global field such that either $\ch(F)>0$ or $F$ is a
totally imaginary number field. Let $X$ be a smooth proper and
geometrically irreducible scheme over $F$ of dimension $d$.
Assume that $i\ge2$. Then
\[
CH^{d+i}(X,j)\simeq
\begin{cases}
\text{uniquely divisible},
    & \mbox{if $2i-j\ge3$},\\
\text{infinite torsion} \oplus \text{divisible},
    & \mbox{if $2i-j=2$},\\
\text{torsion free},
    & \mbox{if $2i-j=-2d$},\\
\text{uniquely divisible},
    & \mbox{if $2i-j<-2d$}.
\end{cases}
\]
\end{thm}
For an arbitrary number field, the same conclusions hold up to a
$2$-torsion obstruction coming from the real places. More precisely,
in the ranges $2i-j\ge3$ and $2i-j<-2d$, the quotient of
$CH^{d+i}(X,j)$ by its subgroup of elements killed by $2$ is uniquely
divisible; see \autoref{cor:global-real}.



Under
the motivic Bass conjecture, the uniquely divisible group above in the range
$2i-j\ge3$ is also finitely generated, and hence it vanishes; see
\autoref{prop:global-Bass}.
\medskip

The paper is organized as follows. In \autoref{sec:prelim}, we recall the
finite-coefficient \'etale cycle class map and the results in characteristic
$p$ used later. In \autoref{sec:main}, we prove the general divisibility and torsion-freeness
results for $CH^{d+i}(X,j)$ and $A^{d+i}(X,j)$. In \autoref{sec:application}, we apply
them to finite fields, local fields, and global fields.

\subsection*{Notation}
Throughout this note, all higher Chow groups $CH^r(X,j)$ are taken
with $j\ge0$ unless explicitly stated otherwise. 
By a \textbf{local field} we mean a non-archimedean local field with finite residue
field. By a \textbf{global field} we mean a finite extension of $\Q$ (a \textbf{number field}) or a 
function field of one variable over a finite field. 
For an abelian group $G$ and an integer $m\ge1$, we write $G[m]$ and
$G/m$ for the kernel and cokernel of multiplication by $m$ on $G$,
respectively. 
For an abelian group $G$ and a prime $l$, we put
\[
G\{l\}:=G[l^\infty]
=
\bigcup_{n\ge1}G[l^n].
\]
For a prime $p$, we put
\[
G\{p'\}:=\bigoplus_{l\neq p}G\{l\}.
\]
Thus $G\{l\}$ is the $l$-primary torsion subgroup of $G$, and
$G\{p'\}$ is its prime-to-$p$ torsion subgroup.
We say that $G$ is \textbf{$p'$-divisible} if it is $l$-divisible for every
prime $l\neq p$, and that $G$ is \textbf{uniquely $p'$-divisible} if it is
uniquely $l$-divisible for every prime $l\neq p$.
For a field $F$, we denote by $\ch(F)$ its characteristic 
and we denote by $\cd(F)$ (resp.~$\cd_l(F)$) the cohomological  dimension (resp.~$l$-cohomological dimension for a prime number $l$) of the absolute Galois group $G_F$ of $F$ (cf.~\cite[Chapter I, Section 3]{SerreCG}).

\subsection*{Acknowledgements} 
The first author was supported by JSPS KAKENHI Grant Number 24K06672.
The second author was supported by JSPS KAKENHI Grant Number 21K03188.

\section{Preliminaries}\label{sec:prelim}

Let $F$ be a field, 
and let $X$ be a separated scheme of finite type over $F$ 
with $d = \dim(X)$.
For integers $r,j\ge 0$, Bloch's \textbf{higher Chow group} is defined by
\[
  CH^r(X,j):=H_j(z^r(X,\bullet)),
\]
where $z^r(X,\bullet)$ is Bloch's cycle complex whose degree $n$ term
is generated by codimension $r$ cycles on $X\times (\mathbb{P}^1\setminus\set{1})^n$,
meeting all faces properly. In particular, $CH^r(X,0)=CH^r(X)$.

Let $m$ be an integer prime to $\ch(F)$, and let $r\ge 0$.
There is a canonical cycle class morphism
\begin{equation}\label{et-cyc}
    \Z/m(r)_M\to R\epsilon_*\Z/m(r),
\end{equation}
where $\Z/m(r)_M$ is the motivic complex of weight $r$ and
$\epsilon\colon X_{\et}\to X_{\Zar}$ is the canonical morphism of sites.
The Beilinson--Lichtenbaum comparison theorem
\cite{GL01,Vo11} asserts that the induced morphism
\begin{equation}\label{qis}
    \Z/m(r)_M\xrightarrow{\simeq}
    \tau_{\le r}R\epsilon_*\Z/m(r)
\end{equation}
is a quasi-isomorphism. Taking Zariski hypercohomology in
\eqref{et-cyc}, we obtain the \'etale cycle class map
\[
cl^r\colon CH^r(X,j;\Z/m)
\to H_\et^{2r-j}(X,\Z/m(r)),
\]
where
\[
CH^r(X,j;\Z/m)
:=
H_j\left(z^r(X,\bullet)\otimes^{\mathbb L}\Z/m\right).
\]
The distinguished triangle
\[
\Z(r)_M\xrightarrow{m}\Z(r)_M\to\Z/m(r)_M\xrightarrow{+1}
\]
induces a short exact sequence
\begin{equation}\label{eq:coeff}
0\to CH^r(X,j)/m \to CH^r(X,j;\Z/m) \to CH^r(X,j-1)[m]\to 0,
\end{equation}
where we put $CH^r(X,-1)=0$.
Suppose that $X$ is smooth of dimension $d$. We shall use the
localization spectral sequence
\begin{equation}\label{eq:loc-ss}
E_1^{a,b}
=
\bigoplus_{x\in X^a}
CH^{r-a}(\kappa(x),-a-b)
\Rightarrow
CH^r(X,-a-b),
\end{equation}
where $X^a$ denotes the set of points of codimension $a$ in $X$, 
and $\kappa(x)$ denotes the residue field at $x$.

The following consequence of the Beilinson--Lichtenbaum comparison
will be used repeatedly (cf.\ \cite[Lemma~7.1]{GKR22} and
\cite[Proposition~4.2]{KZ24}).

\begin{lem}\label{lem:cycl}
Let $X$ be a smooth scheme of finite type over a field $F$, and put
$d=\dim(X)$. Let $l\neq\ch(F)$ be a prime and let $n\ge1$. Then the
\'etale cycle class map
\[
cl^r\colon CH^r(X,j;\Z/l^n)
\to H^{2r-j}_\et(X,\Z/l^n(r))
\]
is an isomorphism if $r\le j$ or $r\ge d+\cd_l(F)$, and it is
injective if $r=j+1$.
\end{lem}

\begin{proof}
By the Beilinson--Lichtenbaum comparison theorem, the cycle class map
is an isomorphism in cohomological degrees at most $r$ and is injective
in degree $r+1$. Since the relevant cohomological degree is $2r-j$,
this proves the assertions for $r\le j$ and $r=j+1$.

It remains to consider the case $r\ge d+\cd_l(F)$. The comparison
quasi-isomorphism \eqref{qis} and the canonical truncation triangle give
a distinguished triangle
\[
\Z/l^n(r)_M
\to R\epsilon_*\Z/l^n(r)
\to \tau_{\ge r+1}R\epsilon_*\Z/l^n(r)
\xrightarrow{+1}.
\]
For every affine open subset $U\subset X$, one has
\[
\cd_l(U)\le \dim(U)+\cd_l(F)\le d+\cd_l(F)\le r
\]
(cf.\ \cite[\href{https://stacks.math.columbia.edu/tag/0F0V}{Tag 0F0V}]{stacks-project}).
Hence
\[
H^q_\et(U,\Z/l^n(r))=0
\qquad(q\ge r+1).
\]
It follows that $R^q\epsilon_*\Z/l^n(r)=0$ for $q\ge r+1$, and therefore
$\tau_{\ge r+1}R\epsilon_*\Z/l^n(r)=0$. Thus the cycle class morphism is
a quasi-isomorphism without truncation, which proves the assertion.
\end{proof}


We also recall the following result in characteristic $p$.

\begin{thm}[Geisser--Levine {\cite{GL00}}]\label{thm:GL}
Let $E$ be a field of characteristic $p>0$. Then $CH^r(E,j)$ is
uniquely $p$-divisible whenever $r\neq j$.
\end{thm}
We shall also use the following standard fact concerning $p$-ranks.

\begin{lem}\label{lem:p-rank-extension}
Let $E/F$ be a finitely generated separable extension of fields of
characteristic $p>0$, and put $q=\trdeg_F(E)$. If $[F:F^p]<\infty$, then
\[
[E:E^p]=p^q[F:F^p].
\]
\end{lem}

\begin{proof}
Choose a separating transcendence basis $t_1,\ldots,t_q$ of $E/F$.
Then $E/F(t_1,\ldots,t_q)$ is finite separable, and hence
\[
\dim_E\Omega^1_{E/\bF_p}
=
\dim_F\Omega^1_{F/\bF_p}+q.
\]
Since the $p$-rank of a field $K$ with $[K:K^p]<\infty$ is
$\dim_K\Omega^1_{K/\bF_p}$, the assertion follows.
\end{proof}

The following elementary consequence of the Bloch--Gabber--Kato
theorem will be used in the proof of the characteristic-primary
divisibility result.

\begin{lem}\label{lem:p-Milnor}
Let $E$ be a field of characteristic $p>0$. Suppose that
$[E:E^p]\le p^{t-1}$. Then $K_n^M(E)$ is uniquely $p$-divisible for
every $n\ge t$.
\end{lem}

\begin{proof}
By Izhboldin's theorem, $K_n^M(E)$ is $p$-torsion free
(\cite[Theorem~6]{Izh00}). The assumption implies
$\dim_E\Omega_E^1\le t-1$, and hence $\Omega_E^n=0$ for $n\ge t$.
By the Bloch--Gabber--Kato theorem, the differential symbol gives an
injection
$K_n^M(E)/p\to\Omega_E^n$
(\cite[Corollary~2.8]{BK86}). Thus $K_n^M(E)/p=0$, which proves the
assertion.
\end{proof}

\section{Divisibility and torsion in higher Chow groups}\label{sec:main}
Throughout this section, 
let $F$ be a field with finite $l$-cohomological dimension $\cd_l(F) = s< \infty$ 
for some prime $l\neq \ch(F)$ 
and a smooth scheme $X$ of dimension $\dim(X) =d$. 
In this section, we establish the general divisibility and torsion-freeness results for
$CH^{d+i}(X,j)$ and for the kernel of the push-forward to the
base field when $X$ is proper. 

\begin{prop}\label{prop:l-div}
Assume that $i\ge s$. 

\begin{enumerate}
    \item If $2i-j > s+1$, $CH^{d+i}(X,j)$ is uniquely $l$-divisible. 
    \item If $2i-j = s+1$, $CH^{d+i}(X,j)$ is $l$-divisible. 
    \item If $2i-j = -2d$, $CH^{d+i}(X,j)$ is $l$-torsion free.
    \item If $2i-j < -2d$, $CH^{d+i}(X,j)$ is uniquely $l$-divisible.
\end{enumerate}
\end{prop}

\begin{proof}
Since $i\ge s$, we have $d+i\ge d+\cd_l(F)$. Hence
\autoref{lem:cycl} gives isomorphisms
\begin{equation}\label{eq:d+i,j}
\begin{split}
CH^{d+i}(X,j;\Z/l)
&\simeq
H^{2d+2i-j}_{\et}(X,\Z/l(d+i)),\\
CH^{d+i}(X,j+1;\Z/l)
&\simeq
H^{2d+2i-j-1}_{\et}(X,\Z/l(d+i)).
\end{split}
\end{equation}
Together with the exact sequences induced by the distinguished triangle, these isomorphisms give
\begin{equation}\label{ex:modl}
0\to CH^{d+i}(X,j)/l
\to H^{2d+2i-j}_{\et}(X,\Z/l(d+i))
\to CH^{d+i}(X,j-1)[l]
\to 0
\end{equation}
and
\begin{equation}\label{ex:modl2}
0\to CH^{d+i}(X,j+1)/l
\to H^{2d+2i-j-1}_{\et}(X,\Z/l(d+i))
\to CH^{d+i}(X,j)[l]
\to 0.
\end{equation}

The $l$-cohomological dimension of $X$ is at most
$2d+\cd_l(F)=2d+s$. If $2i-j\ge s+1$, then
$2d+2i-j\ge2d+s+1$, and hence
$H^{2d+2i-j}_{\et}(X,\Z/l(d+i))=0$. It follows from
\eqref{ex:modl} that $CH^{d+i}(X,j)/l=0$. Thus
$CH^{d+i}(X,j)$ is $l$-divisible. 

If $2i-j\ge s+2$, then $2d+2i-j-1\ge2d+s+1$, and hence
$H^{2d+2i-j-1}_{\et}(X,\Z/l(d+i))=0$. By
\eqref{ex:modl2}, we have $CH^{d+i}(X,j)[l]=0$.
Together with the $l$-divisibility proved above, this shows that
$CH^{d+i}(X,j)$ is uniquely $l$-divisible. 
The assertions (1) and (2) hold.

Next, assume that $2i-j\le-2d$. Since $i\ge s$,
\autoref{lem:cycl} still gives the isomorphisms
\eqref{eq:d+i,j}.
If $2i-j<-2d$, then $2d+2i-j<0$ and
$2d+2i-j-1<0$. Therefore,
\[
H^{2d+2i-j}_{\et}(X,\Z/l(d+i))
=H^{2d+2i-j-1}_{\et}(X,\Z/l(d+i))
=0.
\]
The exact sequences \eqref{ex:modl} and \eqref{ex:modl2} give
$CH^{d+i}(X,j)/l=0$ and $CH^{d+i}(X,j)[l]=0$. Thus
$CH^{d+i}(X,j)$ is uniquely $l$-divisible.

If $2i-j=-2d$, then $2d+2i-j-1=-1$. Hence
$H^{2d+2i-j-1}_{\et}(X,\Z/l(d+i))=0$, and
\eqref{ex:modl2} gives $CH^{d+i}(X,j)[l]=0$.
This proves the assertions (3) and (4).
\end{proof}

\begin{rem}\label{rem:BS}
The negative range $2i-j\le-2d$ in \autoref{prop:l-div} is related to
the Beilinson--Soul\'e vanishing conjecture
(\cite[Conjecture~5]{Kah05}).
If $X$ is regular, the Beilinson--Soul\'e vanishing conjecture predicts
\[
CH^{d+i}(X,j)=0
\qquad\text{if $2i-j<-2d$}.
\]
Therefore, the unique $l$-divisibility proved above may be viewed as
an unconditional prime-to-$\ch(F)$ substitute for this expected
integral vanishing.
\end{rem}

\begin{prop}\label{prop:uniq-p-div}
Suppose that  $\ch(F) = p>0$ and 
$[F:F^p]\le p^{s-1}$. If $i\ge s$, then
$CH^{d+i}(X,j)$ is uniquely $p$-divisible.
\end{prop}

\begin{proof}
Consider the localization spectral sequence \eqref{eq:loc-ss} with
$r=d+i$:
\[
E_1^{a,b}
=
\bigoplus_{x\in X^a}
CH^{d+i-a}(\kappa(x),-a-b)
\Rightarrow
CH^{d+i}(X,-a-b).
\]
It is enough to show that every $E_1$-term is uniquely
$p$-divisible. Indeed, the category of uniquely $p$-divisible
abelian groups is closed under direct sums, kernels, cokernels, and
extensions.
Let $x\in X^a$. If $d+i-a\neq-a-b$, then
$CH^{d+i-a}(\kappa(x),-a-b)$ is uniquely $p$-divisible by
\autoref{thm:GL}.
Suppose that $d+i-a=-a-b$. Then
\[
CH^{d+i-a}(\kappa(x),-a-b)
\simeq
K^M_{d+i-a}(\kappa(x)).
\]
Put $K=\kappa(x)$. Since $X$ is smooth over $F$, the extension
$K/F$ is finitely generated and separable, 
and
$\trdeg_F(K)\le d-a$. Therefore,
\[
[K:K^p]
=
p^{\trdeg_F(K)}[F:F^p]
\le p^{d-a+s-1}
\le p^{d+i-a-1}.
\]
By \autoref{lem:p-Milnor}, $K^M_{d+i-a}(K)$ is uniquely
$p$-divisible.
Thus every $E_1$-term is uniquely $p$-divisible. For a fixed integer
$j$, the spectral sequence induces a finite filtration on
$CH^{d+i}(X,j)$ whose graded pieces are subquotients of the groups
$E_\infty^{a,b}$ with $-a-b=j$. Hence $CH^{d+i}(X,j)$ is uniquely
$p$-divisible.
\end{proof}

When the structure map $f\colon X\to \Spec(F)$ is proper, 
we consider the kernel 
\[
A^{d+i}(X,j)
:=
\Ker\left(
f_*\colon CH^{d+i}(X,j)\to CH^i(F,j)
\right)
\]
of the proper push-forward map $f_*$.

\begin{thm}\label{thm:CH-div}
Let $F$ be a field with $\cd_l(F)=s$, where $l\neq\ch(F)$, and let 
$X$ be a proper smooth and geometrically irreducible scheme over $F$. 
Assume that $i\ge s$. Then the following holds.
\begin{enumerate}
\item
If $2i-j\ge s$, then, for every $n\ge1$, there is a short exact
sequence
\[
0\to A^{d+i}(X,j)[l^n]
\to CH^{d+i}(X,j)[l^n]
\xrightarrow{f_*}
CH^i(F,j)[l^n]
\to0.
\]
In particular, the push-forward map is surjective on the
$l$-primary torsion subgroups.

\item
The group $A^{d+i}(X,j)$ is uniquely $l$-divisible if $2i-j>s$,
it is $l$-divisible if $2i-j=s$, it is $l$-torsion free if
$2i-j=-2d$, and it is uniquely $l$-divisible if $2i-j<-2d$.

\item
If $2i-j=s$, then the push-forward induces an isomorphism
\[
f_*/l^n\colon
CH^{d+i}(X,j)/l^n
\xrightarrow{\simeq}
CH^i(F,j)/l^n
\]
for every $n\ge1$. Moreover, there is a natural surjection
\[
H^s_{\et}\left(F,H^{2d-1}_{\et}\left(X_{F^{\sep}},\Q_l/\Z_l(d+i)\right)\right)
\to
A^{d+i}(X,j)[l^\infty].
\]
\end{enumerate}
\end{thm}

\begin{proof}
Since $i\ge s$, \autoref{lem:cycl} identifies the finite-coefficient
higher Chow groups of $X$ and $F$ with the corresponding \'etale
cohomology groups. 
These identifications are compatible with proper
push-forward on higher Chow groups and the cohomological trace map;
see \cite[Section~3, in particular (3.8)]{GL01} and
\cite[Sections~8.2 and~8.5]{Fu11}.
The exact sequences induced by the distinguished triangle are compatible
with the proper push-forward on higher Chow groups and the \'etale
trace map. Thus, for every $n\ge1$, we have commutative diagrams with
exact rows
\begin{equation}\label{diag:modl}
\vcenter{
\xymatrix@C=2.5mm{
0 \ar[r] &
CH^{d+i}(X,j)/l^n
\ar[d]^{f_*/l^n}
\ar[r] &
H^{2d+2i-j}_{\et}(X,\Z/l^n(d+i))
\ar[d]^{\Tr_f^1}
\ar[r] &
CH^{d+i}(X,j-1)[l^n]
\ar[d]^{f_*[l^n]}
\ar[r] &
0\\
0 \ar[r] &
CH^i(F,j)/l^n
\ar[r] &
H^{2i-j}_{\et}(F,\Z/l^n(i))
\ar[r] &
CH^i(F,j-1)[l^n]
\ar[r] &
0,
}}
\end{equation}
and
\begin{equation}\label{diag:modl2}
\vcenter{
\xymatrix@C=2.5mm{
0 \ar[r] &
CH^{d+i}(X,j+1)/l^n
\ar[d]^{f_*/l^n}
\ar[r] &
H^{2d+2i-j-1}_{\et}(X,\Z/l^n(d+i))
\ar[d]^{\Tr_f^2}
\ar[r] &
CH^{d+i}(X,j)[l^n]
\ar[d]^{f_*[l^n]}
\ar[r] &
0\\
0 \ar[r] &
CH^i(F,j+1)/l^n
\ar[r] &
H^{2i-j-1}_{\et}(F,\Z/l^n(i))
\ar[r] &
CH^i(F,j)[l^n]
\ar[r] &
0.
}}
\end{equation}
We first study the trace map in \eqref{diag:modl2}. Consider the
Hochschild--Serre spectral sequence
\begin{equation}\label{eq:HSSS}
E_2^{a,b}
=
H^a_{\et}\left(F,H^b_{\et}\left(X_{F^{\sep}},\Z/l^n(d+i)
\right)\right) 
\Rightarrow H^{a+b}_{\et}\left(X,\Z/l^n(d+i)\right).
\end{equation}
Geometric Poincar\'e duality \cite[Corollary~8.5.3]{Fu11} gives
\[
H^{2d}_{\et}\left(X_{F^{\sep}},\Z/l^n(d+i)\right)
\simeq
\Z/l^n(i).
\]
Suppose that $2i-j\ge s$. If $2i-j\ge s+2$, then
$H^{2i-j-1}_{\et}(F,\Z/l^n(i))=0$, so $\Tr_f^2$ is
surjective.

Assume that $2i-j=s+1$. 
The relevant total degree in \eqref{eq:HSSS} is $2d+s$. 
The only possibly nonzero term is 
$E_2^{s,2d}\simeq H^s_\et(F,\Z/l^n(i))$. 
The Hochschild--Serre spectral sequence, together with
Poincar\'e duality, gives an isomorphism
\[
\Tr_f^2\colon
H^{2d+s}_{\et}(X,\Z/l^n(d+i))
\xrightarrow{\simeq}
H^s_{\et}(F,\Z/l^n(i)).
\]

Assume that $2i-j=s$. 
In total degree $2d+s-1$, the only possibly
nonzero terms are
$E_2^{s-1,2d}\simeq H^{s-1}_{\et}(F,\Z/l^n(i))$ and $E_2^{s,2d-1}$. 
Every incoming differential to $E_r^{s-1,2d}$ has source in geometric
degree greater than $2d$, and every outgoing differential has target
in Galois degree greater than $s$. Hence $E_\infty^{s-1,2d} = E_2^{s-1,2d}$. 
The filtration on
$H^{2d+s-1}_{\et}(X,\Z/l^n(d+i))$ therefore has the form 
\[
0\subset F^s\subset F^{s-1}
=H^{2d+s-1}_{\et}(X,\Z/l^n(d+i)),
\]
with
$F^s\simeq E_\infty^{s,2d-1}$ and $F^{s-1}/F^s
\simeq E_\infty^{s-1,2d}
\simeq H^{s-1}_{\et}(F,\Z/l^n(i))$.
The trace map 
\[
\Tr_f^2\colon
H^{2d+s-1}_{\et}(X,\Z/l^n(d+i))
\to
H^{s-1}_{\et}(F,\Z/l^n(i))
\]
is the quotient map and is surjective and $\Ker(\Tr_f^2) \simeq E_\infty^{s,2d-1}$.
It follows from \eqref{diag:modl2} that
\[
f_*\colon
CH^{d+i}(X,j)[l^n]
\to
CH^i(F,j)[l^n]
\]
is surjective whenever $2i-j\ge s$. Its kernel is
$A^{d+i}(X,j)[l^n]$, which proves assertion~(1).
Consider the exact sequence
\[
0\to A^{d+i}(X,j)
\to CH^{d+i}(X,j)
\xrightarrow{f_*}
\operatorname{Im}(f_*)
\to0.
\]
Applying the snake lemma to multiplication by $l^n$, assertion~(1)
gives
\begin{equation}\label{ex:A/l}
\begin{split}
0&\to A^{d+i}(X,j)[l^n]
\to CH^{d+i}(X,j)[l^n]
\to CH^i(F,j)[l^n]
\to0,\\
0&\to A^{d+i}(X,j)/l^n
\to CH^{d+i}(X,j)/l^n
\to \operatorname{Im}(f_*)/l^n
\to0.
\end{split}
\end{equation}

We prove assertion~(2). Suppose first that $2i-j\ge s+2$.
By \autoref{prop:l-div}, $CH^{d+i}(X,j)$ is uniquely
$l$-divisible. The second exact sequence in \eqref{ex:A/l} shows that
$A^{d+i}(X,j)$ is $l$-divisible, while its inclusion into
$CH^{d+i}(X,j)$ shows that it is $l$-torsion free. Thus
$A^{d+i}(X,j)$ is uniquely $l$-divisible.

Next, suppose that $2i-j=s+1$. By \autoref{prop:l-div},
$CH^{d+i}(X,j)$ is $l$-divisible. Hence the second exact sequence in
\eqref{ex:A/l} gives $A^{d+i}(X,j)/l=0$. Thus
$A^{d+i}(X,j)$ is $l$-divisible.

Choose a finite separable extension $F'/F$ such that
$X(F')\neq\emptyset$, and put $X_{F'}=X\otimes_FF'$.
Let $\pi\colon X_{F'}\to X$ be the projection. The flat pull-back and
proper push-forward induce maps
\[
\pi^*\colon
A^{d+i}(X,j)[l^\infty]
\to
A^{d+i}(X_{F'},j)[l^\infty]
\]
and
\[
\pi_*\colon
A^{d+i}(X_{F'},j)[l^\infty]
\to
A^{d+i}(X,j)[l^\infty].
\]
Their composite is multiplication by $[F':F]$.

The same argument as above shows that
$A^{d+i}(X_{F'},j)$ is $l$-divisible. Since $X(F')\neq\emptyset$,
the structure morphism of $X_{F'}$ admits a section. Therefore, the
left vertical map in the analogue of \eqref{diag:modl2} over $F'$ is
surjective. The middle vertical map is an isomorphism because
$2i-j=s+1$. A diagram chase shows that
$A^{d+i}(X_{F'},j)[l]=0$. Since this group is $l$-divisible, it follows
that $A^{d+i}(X_{F'},j)[l^\infty]=0$.

The group $A^{d+i}(X,j)[l^\infty]$ is $l$-divisible, so multiplication
by $[F':F]$ is surjective on it. Therefore,
\[
\pi_*\colon
A^{d+i}(X_{F'},j)[l^\infty]
\to
A^{d+i}(X,j)[l^\infty]
\]
is surjective. Hence $A^{d+i}(X,j)[l^\infty]=0$, and
$A^{d+i}(X,j)$ is uniquely $l$-divisible.

Now suppose that $2i-j=s$. In \eqref{diag:modl}, the trace map
\[
\Tr_f^1\colon H^{2d+s}_{\et}(X,\Z/l^n(d+i))
\to H^s_{\et}(F,\Z/l^n(i))
\]
is an isomorphism.
Moreover, $2i-(j-1)=s+1$, so
$A^{d+i}(X,j-1)$ is uniquely $l$-divisible by the case already
proved. Hence $A^{d+i}(X,j-1)[l^n]=0$.
Assertion~(1), applied to $j-1$, shows that the right vertical map in
\eqref{diag:modl} is surjective. It is therefore an isomorphism. A
diagram chase gives an isomorphism
\[
f_*/l^n\colon CH^{d+i}(X,j)/l^n \xrightarrow{\simeq} CH^i(F,j)/l^n.
\]
In particular, $A^{d+i}(X,j)/l^n=0$, so
$A^{d+i}(X,j)$ is $l$-divisible.

Suppose that $2i-j<-2d$. By \autoref{prop:l-div}, both
$CH^{d+i}(X,j)$ and $CH^i(F,j)$ are uniquely $l$-divisible.
Hence their kernel $A^{d+i}(X,j)$ is uniquely $l$-divisible.

Finally, if $2i-j=-2d$, then $CH^{d+i}(X,j)$ is
$l$-torsion free by \autoref{prop:l-div}. Therefore, its subgroup
$A^{d+i}(X,j)$ is also $l$-torsion free. This proves
assertion~(2).

It remains to prove the last assertion of (3). Assume that
$2i-j=s$. Taking the direct limit of \eqref{diag:modl2} over $n$, we
obtain a commutative diagram with exact rows
\begin{equation}\label{diag:modl2-lim}
\xymatrix@C=3mm{
0\ar[r] &
CH^{d+i}(X,j+1)\otimes\Q_l/\Z_l
\ar[r]
\ar[d]^{f_*\otimes\Q_l/\Z_l} &
H^{2d+s-1}_{\et}(X,\Q_l/\Z_l(d+i))
\ar[r]
\ar@{->>}[d]^{\Tr_f} &
CH^{d+i}(X,j)[l^\infty]
\ar[d]^{f_*}
\ar[r] &
0\\
0\ar[r] &
CH^i(F,j+1)\otimes\Q_l/\Z_l
\ar[r] &
H^{s-1}_{\et}(F,\Q_l/\Z_l(i))
\ar[r] &
CH^i(F,j)[l^\infty]
\ar[r] &
0.
}
\end{equation}
By the snake lemma, there is an exact sequence
\begin{equation}\label{ex:modl2-lim-full}
0\to
\Ker(f_*\otimes\Q_l/\Z_l)
\to \Ker(\Tr_f)
\to A^{d+i}(X,j)[l^\infty]
\to \Coker(f_*\otimes\Q_l/\Z_l)
\to 0.
\end{equation}
The Hochschild--Serre filtration in total degree $2d+s-1$ gives
$\Ker(\Tr_f)\simeq E_\infty^{s,2d-1}$. Moreover, there is a natural
surjection
\[
H^s_{\et}\left(F,H^{2d-1}_{\et}\left(X_{F^{\sep}},\Q_l/\Z_l(d+i)\right)\right)
\to E_\infty^{s,2d-1}.
\]
Choose a finite separable extension $F'/F$ such that
$X(F')\neq\emptyset$. Over $F'$, the push-forward
\[
CH^{d+i}(X_{F'},j+1)\otimes\Q_l/\Z_l
\to
CH^i(F',j+1)\otimes\Q_l/\Z_l
\]
is surjective because the structure morphism admits a section.
Therefore, \eqref{ex:modl2-lim-full} over $F'$ gives a surjection
\[
H^s_{\et}\left(F',H^{2d-1}_{\et}\left(X_{F^{\sep}},\Q_l/\Z_l(d+i)\right)\right)
\to
A^{d+i}(X_{F'},j)[l^\infty].
\]
The corestriction map on Galois cohomology and the proper
push-forward on higher Chow groups give a commutative diagram
\[
\xymatrix{
H^s_{\et}\left(F',H^{2d-1}_{\et}\left(X_{F^{\sep}},\Q_l/\Z_l(d+i)\right)\right)
\ar@{->>}[r]
\ar[d]
&A^{d+i}(X_{F'},j)[l^\infty]
\ar@{->>}[d]\\
H^s_{\et}\left(F,H^{2d-1}_{\et}\left(X_{F^{\sep}},\Q_l/\Z_l(d+i)\right)\right)\ar[r]
&A^{d+i}(X,j)[l^\infty].
}
\]
The right vertical map is surjective by the norm argument above.
Hence
\[
H^s_{\et}\left(F,H^{2d-1}_{\et}\left(X_{F^{\sep}},\Q_l/\Z_l(d+i)\right)\right)
\to A^{d+i}(X,j)[l^\infty]
\]
is surjective. This proves assertion~(3).
\end{proof}

\section{Applications to arithmetic fields}\label{sec:application}
Let $X$ be a smooth proper and geometrically irreducible scheme over a
field $F$, and put $d=\dim(X)$.
Before treating arithmetic fields separately, we record an elementary
group-theoretic lemma that will be used repeatedly. 

\begin{lem}[{\cite[Lemma~3.4.4]{RS00}}]\label{lem:RS}
Let $G$ be an abelian group and let $\Sigma$ be a set of primes.
Assume that
\[
\varprojlim_{m\in\mathbb N_\Sigma}G/m
\]
is finite, where $\mathbb N_\Sigma$ denotes the set of positive
integers all of whose prime divisors belong to $\Sigma$. Then there
is a non-canonical decomposition
\[
G\simeq T\oplus D,
\]
where $T$ is finite and its order is divisible only by primes in
$\Sigma$, and $D$ is $l$-divisible for every $l\in\Sigma$.
\end{lem}

\subsection*{Finite fields}

We first consider the case where $F=\bF_q$ is a finite field.
Put $p=\ch(\bF_q)$. 

\begin{lem}\label{lem:fin}
For every $i\in\Z$ and every $j\ge0$, we have
\[
CH^i(\bF_q,j)\simeq
\begin{cases}
\Z, & \mbox{if $(i,j)=(0,0)$},\\
\Z/(q^i-1), & \mbox{if $i\ge1$ and $2i-j=1$},\\
0, & \mbox{otherwise}.
\end{cases}
\]
\end{lem}

\begin{proof}
If $i<0$, then $z^i(\bF_q,\bullet)=0$. Moreover,
$CH^0(\bF_q,0)=\Z$ and $CH^0(\bF_q,j)=0$ for $j>0$. Thus it remains
to consider $i\ge1$.

If $j=0$, then $CH^i(\bF_q,0)=0$ for dimension
reasons. If $j>0$, \cite[Theorem~3.1]{Lev94} gives an injection
$CH^i(\bF_q,j)\otimes_\Z\Q\to K_j(\bF_q)\otimes_\Z\Q$.
Since $K_j(\bF_q)$ is finite for $j>0$
(\cite[Chapter~IV, Corollary~1.13]{Wei13}), the group
$CH^i(\bF_q,j)$ is torsion. Hence $CH^i(\bF_q,j)$ is torsion for every
$i\ge1$ and $j\ge0$.

Since $[\bF_q:\bF_q^p]=1$, \autoref{prop:uniq-p-div}, applied with
$s=1$, shows that $CH^i(\bF_q,j)$ is uniquely $p$-divisible. In
particular, it has no non-trivial $p$-primary torsion.

Let $l\neq p$ be a prime. Since
$\cd_l(\bF_q)=1$ (\cite[Chapter~I, Sections~2.2 and~3.3]{SerreCG}),
\autoref{prop:l-div}, applied with $d=0$ and $s=1$, shows that
$CH^i(\bF_q,j)$ is uniquely $l$-divisible if $2i-j\ge3$ or $2i-j<0$, and
is $l$-torsion free if $2i-j=0$. Its $l$-primary torsion therefore
vanishes in these ranges. Since this holds for every $l\neq p$, while
the $p$-primary torsion also vanishes, the torsion property proved above
gives $CH^i(\bF_q,j)=0$ if $2i-j\ge3$ or $2i-j\le0$.

It remains to treat $2i-j=1$ and $2i-j=2$. For $l\neq p$ and $n\ge1$, the
distinguished triangle
\[
\Z(i)_M\overset{l^n}{\to}\Z(i)_M\to\Z/l^n(i)_M\to\Z(i)_M[1]
\]
and \autoref{lem:cycl} give an exact sequence
\begin{equation}\label{seq:HFq}
0\to CH^i(\bF_q,j+1)/l^n
\to H^{2i-j-1}_{\et}(\bF_q,\Z/l^n(i))
\to CH^i(\bF_q,j)[l^n]\to0.
\end{equation}

Suppose first that $2i-j=1$, so that $j=2i-1$. By the case $2i-j=0$ already
proved, $CH^i(\bF_q,2i)=0$. Hence \eqref{seq:HFq} gives
\[
CH^i(\bF_q,2i-1)[l^n]
\simeq
H^0_{\et}(\bF_q,\Z/l^n(i))
\simeq
\Ker\left(q^i-1\colon\Z/l^n\to\Z/l^n\right).
\]
The last group is cyclic of order $\gcd(l^n,q^i-1)$. Letting $n$ vary,
we see that the $l$-primary part of $CH^i(\bF_q,2i-1)$ is cyclic of
order $l^{v_l(q^i-1)}$. Together with the absence of $p$-primary
torsion, this gives
\[
CH^i(\bF_q,2i-1)\simeq\Z/(q^i-1).
\]

Suppose next that $2i-j=2$, so that $j=2i-2$. The exact sequence
\eqref{seq:HFq} becomes
\[
0\to CH^i(\bF_q,2i-1)/l^n
\to H^1_{\et}(\bF_q,\Z/l^n(i))
\to CH^i(\bF_q,2i-2)[l^n]\to0.
\]
Since $G_{\bF_q}\simeq\widehat{\Z}$ is procyclic, the standard
calculation of its cohomology gives
\[
H^1_{\et}(\bF_q,\Z/l^n(i))
\simeq
\Coker\left(q^i-1\colon\Z/l^n\to\Z/l^n\right)
\]
(\cite[Chapter~I, Section~2.2]{SerreCG}). This group and
$CH^i(\bF_q,2i-1)/l^n$ have the same order. Hence the first map in the
above exact sequence is an isomorphism, and
$CH^i(\bF_q,2i-2)[l^n]=0$. This holds for every $l\neq p$ and every
$n\ge1$, while the $p$-primary torsion also vanishes by
\autoref{prop:uniq-p-div}. Since $CH^i(\bF_q,2i-2)$ is torsion, it is
zero.
\end{proof}

\begin{lem}\label{lem:fin-structure}
Suppose $i\ge1$. Then $CH^{d+i}(X,j)_{\mathrm{tor}}$ is finite of
order prime to $p$, and there is a canonical short exact sequence
\[
0\to CH^{d+i}(X,j)_{\mathrm{tor}}
\to CH^{d+i}(X,j)
\to D\to0,
\]
where
\[
D:=CH^{d+i}(X,j)/CH^{d+i}(X,j)_{\mathrm{tor}}
\]
is uniquely divisible. This sequence splits non-canonically. Moreover,
\[
CH^{d+i}(X,j)_{\mathrm{tor}}
=CH^{d+i}(X,j)\{p'\}.
\]
\end{lem}

\begin{proof}
Write
\[
(\Q/\Z)'=\bigoplus_{l\neq p}\Q_l/\Z_l.
\]
Kahn's finiteness theorem \cite[Theorem~1]{Kah03} (see also \cite[Theorem~7.4]{GKR22}) gives the finiteness of
$H^b_{\et}\left(X,(\Q/\Z)'(d+i)\right)$
for every integer $b$.

Let $m$ be an integer prime to $p$. Since $i\ge 1$,
\autoref{lem:cycl} and \eqref{eq:coeff} give an injection
\[
CH^{d+i}(X,j)/m
\to
H^{2(d+i)-j}_{\et}\left(X,\Z/m(d+i)\right).
\]
The exact triangle
\[
\Z/m(d+i)\to(\Q/\Z)'(d+i)
\overset{m}{\to}(\Q/\Z)'(d+i)
\]
shows that the orders of $H^{2(d+i)-j}_{\et}(X,\Z/m(d+i))$ are bounded
independently of $m$. Hence the orders of $CH^{d+i}(X,j)/m$ are also
uniformly bounded.

Applying \eqref{eq:coeff} to $j+1$ and passing to the direct limit over
integers $m$ prime to $p$, we obtain a surjection
\[
H^{2(d+i)-j-1}_{\et}\left(X,(\Q/\Z)'(d+i)\right)
\to CH^{d+i}(X,j)\{p'\}.
\]
Thus $CH^{d+i}(X,j)\{p'\}$ is finite. 

The uniform boundedness above implies that
\[
\varprojlim_{m\in\mathbb N_\Sigma} CH^{d+i}(X,j)/m
\]
is finite for $\Sigma=\{l\mid l\neq p\}$. Applying
\autoref{lem:RS} with this set $\Sigma$, we obtain a non-canonical
decomposition
\[
CH^{d+i}(X,j)\simeq G\oplus D,
\]
where $G$ is a finite group of order prime to $p$ and $D$ is
$p'$-divisible.
Since $[\bF_q:\bF_q^p]=1$, \autoref{prop:uniq-p-div} shows that
$CH^{d+i}(X,j)$ is uniquely $p$-divisible. 
Since $G$ is finite of order prime to $p$, 
$D$ is divisible. 

Since $CH^{d+i}(X,j)$ is $p$-torsion free, we have 
\[
CH^{d+i}(X,j)_{\mathrm{tor}} = CH^{d+i}(X,j)\{p'\}.
\]
The group on the right is finite. 
Since $D$ is a direct summand of $CH^{d+i}(X,j)$, 
its torsion subgroup $D_{\mathrm{tor}}$ is also finite.
On the other hand, 
$D$ is divisible, and the torsion subgroup of a divisible group is divisible. 
Hence $D_{\mathrm{tor}}$ is both finite and divisible, and therefore
$D_{\mathrm{tor}}=0$.
Thus $D$ is torsion free and hence uniquely divisible.

Thus $G=CH^{d+i}(X,j)_{\mathrm{tor}}$. The decomposition above
identifies the canonical quotient by the torsion subgroup with a
uniquely divisible group, and proves the statement.
\end{proof}

We write $A^{d+i}(X,j)$ for the kernel of $f_*$.
\[
A^{d+i}(X,j)
=
\Ker\left(f_*\colon CH^{d+i}(X,j)\to CH^i(\bF_q,j)\right).
\]
\begin{cor}\label{cor:fin2}
Suppose $i\ge1$.
\begin{enumerate}
\item If $2i-j=1$, there is a canonical short exact sequence
\[
0\to A^{d+i}(X,j)
\to CH^{d+i}(X,j)
\xrightarrow{f_*}\Z/(q^i-1)\to0,
\]
where $A^{d+i}(X,j)$ is uniquely divisible. This sequence splits
non-canonically.
\item If $-2d<2i-j<1$, the torsion subgroup
$CH^{d+i}(X,j)_{\mathrm{tor}}$ is finite of order prime to $p$, and
there is a canonical short exact sequence
\[
0\to CH^{d+i}(X,j)_{\mathrm{tor}}
\to CH^{d+i}(X,j)
\to D\to0,
\]
where
\[
D:=CH^{d+i}(X,j)/CH^{d+i}(X,j)_{\mathrm{tor}}
\]
is uniquely divisible. This sequence splits non-canonically.
\item In all other cases, $CH^{d+i}(X,j)$ is uniquely divisible.
\end{enumerate}
\end{cor}
\begin{proof}

    \noindent 
(1)\ Assume that $2i-j=1$. Fix a prime $l\neq p$.
By \autoref{thm:CH-div}\,(3), applied with $s=1$, 
there is a surjection
\[
H^1_{\et}\left(
\bF_q,
H^{2d-1}_{\et}
\left(
X_{\overline{\bF}_q},
\Q_l/\Z_l(d+i)
\right)
\right)
\twoheadrightarrow
A^{d+i}(X,j)[l^\infty].
\]

Put $\overline X=X_{\overline{\bF}_q}$ and
\begin{align*}
T_l &=  H^{2d-1}_{\et}\left(\overline X, \Z_l(d+i)\right),    \\
V_l &=H^{2d-1}_{\et}\left(\overline X,\Q_l(d+i)\right),\\
M_l &=
H^{2d-1}_{\et}
\left(\overline X,\Q_l/\Z_l(d+i)\right).
\end{align*}
Here and below, $H^r_{\et}(-,\Z_l(t))$ denotes continuous $l$-adic
cohomology. 
By the finiteness theorem for \'etale cohomology,
$T_l$ is a finitely generated
$\Z_l$-module (\cite[Theorem~19.2]{MilneEC}).  
The exact sequence
$0\to\Z_l(d+i)\to\Q_l(d+i)\to\Q_l/\Z_l(d+i)\to0$
therefore gives an exact sequence
\[
T_l \to V_l \to M_l \to H^{2d}_{\et}\left(\overline X,\Z_l(d+i)\right) 
\to H^{2d}_{\et}\left(\overline X,\Q_l(d+i)\right).
\]
Let $L_l$ be the image of
$T_l$ in $V_l$. 
Since the kernel of the natural map
\[
T_l\to T_l\otimes_{\Z_l}\Q_l\simeq V_l
\]
is the torsion subgroup $(T_l)_{\mathrm{tor}}$, we have
$L_l\simeq T_l/(T_l)_{\mathrm{tor}}$. 
In particular, $L_l$ is a finite free $\Z_l$-module and
$L_l\otimes_{\Z_l}\Q_l\simeq V_l$.
Thus $L_l$ is a $\Z_l$-lattice in $V_l$. Moreover, it is
Frobenius-stable because the map $T_l\to V_l$ is Galois equivariant.
Then $L_l$ is a Frobenius-stable $\Z_l$-lattice and $M_l\simeq V_l/L_l$.
Since $X$ is geometrically connected, 
the trace isomorphism gives
$H^{2d}_{\et}(\overline X,\Z_l(d+i))\simeq\Z_l(i)$ 
and
$H^{2d}_{\et}
\left(
\overline X,\Q_l(d+i)
\right)
\simeq \Q_l(i)$ 
(\cite[Chapter~VI, Theorem~11.1]{Mil80}).
Under these identifications, the map between the two groups is the
natural injection 
$\Z_l(i)\hookrightarrow\Q_l(i)$.
It follows from the above long exact sequence that the map
$V_l\to M_l$ is surjective and has kernel $L_l$. Hence
$M_l\simeq V_l/L_l$.

Deligne's theorem on weights
\cite[Th\'eor\`eme~3.3.1]{Del80}, applied to
$f\colon X\to\Spec(\bF_q)$ and the pure sheaf $\Q_l(d+i)$ of weight
$-2(d+i)$, shows that $V_l$ is mixed of weights at most
$-2(d+i)+(2d-1)=-2i-1$. In particular, $1$ is not an eigenvalue of
Frobenius on $V_l$.
Let $\varphi$ denote the geometric Frobenius, which is a
topological generator of $G_{\bF_q}$. 
The endomorphism $\varphi-1$ is an automorphism of $V_l$ and induces a
surjection on $M_l\simeq V_l/L_l$. Since $G_{\bF_q}$ is procyclic,
\[
H^1_{\et}(\bF_q,M_l)
\simeq M_l/(\varphi-1)M_l
=0
\]
(\cite[Chapter~I, Section~2.2]{SerreCG}). Hence
$A^{d+i}(X,j)[l^\infty]=0$. Since $A^{d+i}(X,j)$ is $l$-divisible by
\autoref{thm:CH-div}\,(2), it is uniquely $l$-divisible.

By \autoref{prop:uniq-p-div}, $CH^{d+i}(X,j)$ is uniquely
$p$-divisible. By \autoref{lem:fin}, its target is
$CH^i(\bF_q,j)\simeq\Z/(q^i-1)$, whose order is prime to $p$.
It follows that $A^{d+i}(X,j)$ is $p$-torsion free. It is also
$p$-divisible. Indeed, if $a\in A^{d+i}(X,j)$ and $a=pb$ for some
$b\in CH^{d+i}(X,j)$, then $pf_*(b)=0$. Since multiplication by $p$
is an automorphism of $\Z/(q^i-1)$, we have $f_*(b)=0$. Thus
$b\in A^{d+i}(X,j)$. Consequently, $A^{d+i}(X,j)$ is uniquely
$p$-divisible, and hence uniquely divisible.

For every prime $l$ dividing $q^i-1$, \autoref{thm:CH-div}\,(1) shows
that $f_*$ is surjective on the $l$-primary torsion subgroups.
Therefore $f_*$ is surjective, and there is an exact sequence
\[
0\to A^{d+i}(X,j)
\to CH^{d+i}(X,j)
\overset{f_*}{\to}\Z/(q^i-1)
\to0.
\]
A divisible abelian group is injective
(\cite[Section~2.3]{Wei94}), so this sequence splits non-canonically.
Thus
\[
CH^{d+i}(X,j)
\simeq
\Z/(q^i-1)\oplus A^{d+i}(X,j),
\]
where $A^{d+i}(X,j)$ is uniquely divisible.

    \smallskip 
    \noindent 
(2)\ The intermediate case  $-2d < 2i-j<1$ follows from \autoref{lem:fin-structure}.

    \smallskip 
    \noindent 
(3)\ Suppose first that $2i-j>1$. By \autoref{lem:fin},
$CH^i(\bF_q,j)=0$, and hence
$CH^{d+i}(X,j)=A^{d+i}(X,j)$. For every prime $l\neq p$,
\autoref{thm:CH-div}\,(2), applied with $s=1$, shows that
$A^{d+i}(X,j)$ is uniquely $l$-divisible. Moreover,
\autoref{prop:uniq-p-div}, applied with $s=1$, shows that
$CH^{d+i}(X,j)$ is uniquely $p$-divisible. Therefore,
$CH^{d+i}(X,j)$ is uniquely divisible.

    \smallskip 
    \noindent 
Finally, suppose that  $2i-j\le-2d$. 
If $2i-j<-2d$, then \autoref{prop:l-div} and
\autoref{prop:uniq-p-div} show that $CH^{d+i}(X,j)$ is uniquely
$l$-divisible for every prime $l$. Therefore it is uniquely divisible.
If $2i-j=-2d$, then
\autoref{prop:l-div} and \autoref{prop:uniq-p-div} show that the group
is torsion free, while \autoref{lem:fin-structure} shows that it is
a direct sum of a finite group and a uniquely divisible group. Since the whole group is torsion free, the finite summand is zero,
so it is uniquely divisible.
\end{proof}



We shall use the following form of Parshin's conjecture.

\begin{conj}[{Parshin's conjecture, cf.~\cite[Conjecture~16]{Gei19}}]
\label{conj:P}
Let $Y$ be a smooth projective scheme over a finite field. Then
$CH^r(Y,j)\otimes_\Z\Q=0$ for every $r\in\Z$ and every $j>0$.
\end{conj}

\begin{cor}\label{cor:fin}
Assume that $X$ is projective and that Parshin's conjecture holds for $X$.
Suppose $i\ge 1$. Then 
\[
CH^{d+i}(X,j)\simeq
\begin{cases}
\Z/(q^i-1), & \mbox{if $2i-j=1$},\\
\text{finite of order prime to $p$}, & \mbox{if $-2d<2i-j<1$},\\
0, & \mbox{otherwise}.
\end{cases}
\]
\end{cor}

\begin{proof}
Suppose first that $2i-j>1$. If $j=0$, then
$CH^{d+i}(X,0)=0$ for dimension reasons. If $j>0$, then
\autoref{cor:fin2} shows that $CH^{d+i}(X,j)$ is uniquely divisible,
while Parshin's conjecture shows that
$CH^{d+i}(X,j)\otimes_\Z\Q=0$. Therefore $CH^{d+i}(X,j)=0$.

Suppose next that $2i-j=1$. Then $j=2i-1>0$. By
\autoref{cor:fin2}, there is a decomposition
$CH^{d+i}(X,j)\simeq\Z/(q^i-1)\oplus D$, where $D$ is uniquely
divisible. Since $D$ is a $\Q$-vector space, Parshin's conjecture gives
$D=0$. Thus $CH^{d+i}(X,j)\simeq\Z/(q^i-1)$.

Suppose that $-2d<2i-j<1$. By \autoref{cor:fin2}, there is an exact
sequence
\[
0\to G\to CH^{d+i}(X,j)\to D\to0,
\]
where $G$ is finite of order prime to $p$ and $D$ is uniquely
divisible. Since $j>0$, Parshin's conjecture gives
$CH^{d+i}(X,j)\otimes_\Z\Q=0$, and hence $D=0$.

If $2i-j\le-2d$, then \autoref{cor:fin2} shows that
$CH^{d+i}(X,j)$ is uniquely divisible. Again $j>0$, so Parshin's
conjecture forces this group to vanish.
\end{proof}

Parshin's conjecture is known for smooth proper curves over finite
fields. Indeed, Harder's theorem shows that $K_n(C)$ is finite for
every $n>0$; see \cite[Chapter~VI, Theorem~6.1]{Wei13}. 

\begin{prop}\label{prop:fin-curve}
    Let $C$ be a smooth proper and geometrically irreducible
    curve over $\bF_q$. Let $i\ge 1$ and $j\ge  0$. Then
    \[
    CH^{i+1}(C,j)
    \simeq
    \begin{cases}
    \Z/(q^i-1), & 2i-j=1,\\
    K_{2i}(C), & 2i-j=0,\\
    \Z/(q^{i+1}-1), & 2i-j=-1,\\
    0, & \text{otherwise}.
    \end{cases}
    \]
\end{prop}
Here, the Quillen $K$-group $K_{2i}(C)$ is finite of order prime to $p = \ch(\bF_q)$. 
\begin{proof}
    We consider a spectral sequence 
    \begin{equation}\label{eq:motivic-ss}
    E_2^{a,b}=CH^{-b}(C,-a-b)
    \Rightarrow K_{-a-b}(C) 
    \end{equation}
    (\cite[Theorem~13.6]{FS02}). 

    By \autoref{cor:fin} and Parshin's conjecture for curves (\cite[Chapter~VI, Theorem~6.1]{Wei13}), 
    we have 
    $CH^{i+1}(C,j) = 0$ if 
    $2i-j\ge 2$ or $2i-j\le -2$. 
    It is also proved that 
    $CH^{i+1}(C,j) \simeq \Z/(q^i-1)$ if $2i-j=1$. 
    
    Furthermore, by $\Z(1) \simeq \Gm[-1]$, we have 
    $CH^{1}(C,j) = 0$ if $j\ge 2$. 

    \smallskip
    \noindent
    (\textbf{Case $2i-j=0$})\ 
    Let $i\ge1$ and $2i-j=0$. 
    In total degree $2i$ in \eqref{eq:motivic-ss}, 
    the only possibly nonzero $E_2$-term is
    $E_2^{1-i,-i-1}=CH^{i+1}(C,2i)$. The vanishing established above
    also shows that every possible source and target of a differential
    meeting $E_2^{1-i,-i-1}$ is zero. Hence
    \[
    E_\infty^{1-i,-i-1}
    = E_2^{1-i,-i-1}
    = CH^{i+1}(C,2i),
    \]
    and every other associated graded piece of $K_{2i}(C)$ is zero.
    We obtain a canonical isomorphism
    \[
    CH^{i+1}(C,2i)
    \overset{\simeq}{\to}
    K_{2i}(C).
    \]

    \smallskip
    \noindent
    (\textbf{Case $2i-j=-1$})\ 
    Consider the case $2i-j= -1$. 
    Choose a zero-cycle of degree one on $C$. This choice gives a
    non-canonical decomposition in the category of geometric motives
    $DM_{gm}(\bF_q)$:
    \begin{equation}\label{eq:geometric-motive-decomposition}
    M(C)\simeq \Z\oplus M^1(C)\oplus \Z(1)[2].
    \end{equation}
    Using \eqref{eq:geometric-motive-decomposition}, we obtain
    \begin{align}
    &CH^{i+1}(C,2i+1)
    =H^1_M(C,\Z(i+1)) \notag\\
    &\quad \simeq
    H^1_M(\bF_q,\Z(i+1))
    \oplus
    \Hom_{DM_{gm}(\bF_q)}\bigl(M^1(C),\Z(i+1)[1]\bigr) \oplus H^{-1}_M(\bF_q,\Z(i)).
    \label{eq:minus-one-motive}
    \end{align}
    Since $H^{-1}_M(\bF_q,\Z(i)) \simeq CH^i(\bF_q,2i+1) = 0$ 
    and $H^1_M(\bF_q,\Z(i+1)) \simeq CH^{i+1}(\bF_q,2i+1) \simeq \Z/(q^{i+1}-1)$  (\autoref{lem:fin}), 
    we have 
    \begin{equation}\label{eq:minus-one-U}
    CH^{i+1}(C,2i+1)
    \simeq
    \Z/(q^{i+1}-1)\oplus U_i,
    \end{equation}
    where
    $U_i:=
    \Hom_{DM_{gm}(\bF_q)}\bigl(M^1(C),\Z(i+1)[1]\bigr)$.
    Note that $CH^{i+1}(C,2i+1)$ and hence $U_i$ are finite of order prime to $p = \ch(\bF_q)$ (\autoref{cor:fin}). 
    It remains to prove that $U_i=0$. Consider the spectral sequence \eqref{eq:motivic-ss} above. 
    In total degree $2i+1$, 
    the only possibly nonzero terms are 
    \[
    E_2^{-i,-i-1} = CH^{i+1}(C,2i+1),\ \mbox{and}\ E_2^{-i+1,-i-2} = CH^{i+2}(C,2i+1) \simeq \Z/(q^{i+1}-1).
    \]
    Here, the last isomorphism follows from the case of
    $2(i+1)-(2i+1)=1$. The vanishing established above shows that no
    nonzero differential enters or leaves either of these two terms.
    Hence they survive to $E_\infty$, and $K_{2i+1}(C)$ has a finite
    filtration whose two nonzero associated graded pieces are
    $CH^{i+1}(C,2i+1)$ 
    and
    $\Z/(q^{i+1}-1)$.
    Hence
    \begin{equation}\label{eq:odd-order-product}
    |K_{2i+1}(C)| =|CH^{i+1}(C,2i+1)|(q^{i+1}-1).    
    \end{equation}
    
    On the other hand, the computation of the odd $K$-groups (\cite[Chapter VI, Theorem~6.7]{Wei13}) 
    and Quillen's computation (\cite[Chapter~IV, Corollary~1.13]{Wei13}) give
    \[
    K_{2i+1}(C)
    \simeq
    K_{2i+1}(\bF_q)\oplus K_{2i+1}(\bF_q) \simeq  \Z/(q^{i+1}-1)^{\oplus 2}.
    \]
    Combining this equality with 
    \eqref{eq:odd-order-product}, 
    we have $|CH^{i+1}(C,2i+1)| = q^{i+1}-1$ and 
    \eqref{eq:minus-one-U} says $|U_i| =1$ and hence $U_i=0$. 
    We obtain 
    \[
    CH^{i+1}(C,2i+1) \simeq \Z/(q^{i+1}-1).
    \]
\end{proof}



\subsection*{Local fields}

Let $F$ be a local field with residue field $k$, and $p = \ch(k)$.
Moore's theorem \cite[Chapter~IX, Theorem 4.3]{FV02} and Merkur'ev's theorem \cite[Chapter~IX, Theorem 4.7]{FV02} give the following decomposition of $CH^2(F,2)\simeq K_2^M(F)$:
\[
CH^2(F,2)\simeq K_2^M(F) \simeq \mu(F)\oplus D
\]
where $\mu(F)$ is the group of roots of unity in $F$, and $D$ is uniquely divisible.
For $i\ge 3$, it is known that $K_i^M(F)$ is uniquely divisible by Sivitskii's theorem \cite[Chapter~IX, Theorem 4.11]{FV02}. 
We will see that an analogous property holds for the higher Chow groups of varieties over $F$ (\autoref{lem:CH-loc2}, \autoref{cor:CH-loc}).

We first prepare a lemma on \'etale cohomology of $F$.
\begin{lem}\label{lem:coh-loc}
    Let $F$ be a local field with residue field $k=\bF_q$ and $p = \ch(k)$.
    Assume $i\neq 0,1.$
    Then, for every prime $l\neq p$, we have
    $$
    H^2_\et(F,\Q_l/\Z_l(i))=0,\  
    H^1_\et(F,\Q_l/\Z_l(i))\simeq \Z/l^{c_{l,i-1}},\ 
    H^0_\et(F,\Q_l/\Z_l(i))\simeq \Z/l^{c_{l,i}} 
    $$
    where $c_{l,i}=v_l(q^i-1)$.
    If $\ch(F)=0$, then we have
    $$
    H^2_\et(F,\Q_p/\Z_p(i))=0,\  
    H^1_\et(F,\Q_p/\Z_p(i))= C \oplus (\Q_p/\Z_p)^{[F:\Q_p]},\ 
    H^0_\et(F,\Q_p/\Z_p(i))\simeq \Z/p^{c_{p,i}} 
    $$
    for some $c_{p,i}\ge 0$ and a finite group $C$. 
\end{lem}
\begin{proof}
    Take $l\neq p$.
    Then we have
    $$
    H^0_\et(F, \Z/l^n(i))\simeq H^0_\et (k,\Z/l^n(i))\simeq \Ker(q^i -1\colon \Z/l^n \to \Z/l^n).
    $$
    Here the last isomorphism follows from the proof of \autoref{lem:fin}.
    Put $c_{l,i}=v_l(q^i-1)$.
    This gives
    $$
    H^0_\et(F,\Q_l/\Z_l(i))\simeq \varinjlim_n\Ker(q^i -1\colon \Z/l^n \to \Z/l^n)\simeq \Z/l^{c_{l,i}}.
    $$

    By local Tate duality, 
    \[ 
     H_{\et}^{2}(F,\Q_l/\Z_l(i))^\vee \simeq H^0_\et(F,\Z_l(1-i))\simeq \varprojlim_n \Ker(q^{i-1} -1\colon \Z/l^n \to \Z/l^n)=0. 
    \]

    Since $l\neq p$, there is a short exact sequence
    $$
    0\to H^1_\et(k,\Q_l/\Z_l(i)) \to H^1_\et(F,\Q_l/\Z_l(i)) \to H^0_\et(k,\Q_l/\Z_l(i-1))\to 0.
    $$
    The left term is isomorphic to the cokernel of $\operatorname{Frob}_q-1: \Q_l/\Z_l(i)\to \Q_l/\Z_l(i)$ which is zero.
    Thus we have
    \begin{equation}
        H^1_\et(F,\Q_l/\Z_l(i))\simeq H^0_\et(k,\Q_l/\Z_l(i-1))\simeq \Z/l^{c_{l,i-1}}.
    \end{equation}
    This completes the proof in the case $l\neq p$.

    Now we assume that $\ch(F)=0$ and $l=p$.
    The absolute Galois group $G_F$ acts on $\Z/p^n(1) = \mu_{p^n}$ 
    by the $p$-adic cyclotomic character $\chi_p\colon G_F \to \Z_p^\times$. 
    The $p$-adic cyclotomic character has open image: $\chi_p(G_F) \subset \Z_p^\times$ 
    is an open subgroup. 
    Choose $e\ge1$ such that $1+p^e\Z_p\subset\chi_p(G_F)$. 
    Choose $\sigma_0\in G_F$ such that $\chi_p(\sigma_0) = 1+ p^e$. 
    Then 
    $c_{p,i}' := v_p(\chi_p(\sigma_0)^{i} -1) = v_p((1+p^e)^{i}-1) < \infty$, and the order of $H^0_\et(F,\Z/p^n(i))$ is bounded by $p^{c_{p,i}'}$.
    As in the case $l\neq p$, we have
    $$
    H^0_\et(F,\Q_p/\Z_p(i))\simeq \Z/p^{c_{p,i}}
    ,\ H_\et^{2}(F,\Q_p/\Z_p(i))=0 
    $$
    for some $c_{p,i}\leq c'_{p,i}$.
    By local Tate duality, we have
    $$
    H^1_\et(F,\Q_p/\Z_p(i))\simeq H^1_\et(F,\Z_p(1-i))^\vee.
    $$
    From the Euler characteristic formula for $\Q_p(1-i)$, we have
    $$
    \dim_{\Q_p}H^0(F,\Q_p(1-i))-\dim_{\Q_p}H^1_\et(F,\Q_p(1-i))+\dim_{\Q_p}H^2_\et(F,\Q_p(1-i))=-[F:\Q_p].
    $$
    This gives $\dim_{\Q_p}H^1_\et(F,\Q_p(1-i))=[F:\Q_p]$, and hence we have
    $$
    H^1_\et(F,\Q_p/\Z_p(i))\simeq H^1_\et(F,\Z_p(1-i))^\vee\simeq (\text{finite})\oplus \left(\Q_p/\Z_p\right)^{[F:\Q_p]}.
    $$
\end{proof}

We determine the structure of the higher Chow group of a local field $F$.

\begin{lem}\label{lem:CH-loc2}
Let $F$ be a local field with residue field $k$, and put $p=\ch(k)$.
Assume that $i\ge2$. Then the following hold.
\begin{enumerate}
\item If $2i-j=2$ and $\ch(F)=0$, then there is a canonical short
exact sequence
\[
0\to CH^i(F,j)_{\mathrm{tor}}
\to CH^i(F,j)
\to U\to0,
\]
where
\[
U:=CH^i(F,j)/CH^i(F,j)_{\mathrm{tor}}
\]
is uniquely divisible, and there is a non-canonical isomorphism
\[
CH^i(F,j)_{\mathrm{tor}}
\simeq G\oplus(\Q_p/\Z_p)^r
\]
for some finite group $G$ and some integer
$0\le r\le[F:\Q_p]$. The short exact sequence splits
non-canonically.
\item If $2i-j=1$ and $\ch(F)=0$, then
$CH^i(F,j)_{\mathrm{tor}}$ is finite. Moreover,
$CH^i(F,j)\{p'\}$ is finite, and there is a canonical short exact
sequence
\[
0\to CH^i(F,j)\{p'\}
\to CH^i(F,j)
\to D\to0,
\]
where
\[
D:=CH^i(F,j)/CH^i(F,j)\{p'\}
\]
is uniquely $p'$-divisible. This sequence splits non-canonically.
\item If $2i-j\in\{1,2\}$ and $\ch(F)>0$, then
$CH^i(F,j)_{\mathrm{tor}}$ is finite of order prime to $p$, and there
is a canonical short exact sequence
\[
0\to CH^i(F,j)_{\mathrm{tor}}
\to CH^i(F,j)
\to D\to0,
\]
where
\[
D:=CH^i(F,j)/CH^i(F,j)_{\mathrm{tor}}
\]
is uniquely divisible. This sequence splits non-canonically.
\item In all other cases, $CH^i(F,j)$ is uniquely divisible.
\end{enumerate}
\end{lem}

\begin{proof}
\smallskip
\noindent
(1) Assume that $2i-j=2$ and $\ch(F)=0$. For every prime $l$, as in
\eqref{ex:modl}, there is an injection
\[
CH^i(F,j)/l^n\hookrightarrow H^2_{\et}(F,\Z/l^n(i)).
\]
Passing to the inverse limit gives an injection
\[
\varprojlim_n CH^i(F,j)/l^n\hookrightarrow H^2_{\et}(F,\Z_l(i)).
\]
Here and below, $H^r_{\et}(-,\Z_l(t))$ denotes continuous $l$-adic
cohomology. Since the cohomology groups with finite coefficients are
finite, the relevant inverse systems satisfy the Mittag--Leffler
condition. By local Tate duality (\cite[Chapter VII, Theorem~7.2.9]{NSW08}),
\[
H^2_{\et}(F,\Z_l(i))
\simeq H^0_{\et}(F,\Q_l/\Z_l(1-i))^\vee,
\]
which is finite and trivial for all but finitely many primes by
\autoref{lem:coh-loc}. 
Here $\vee$ denotes the Pontryagin dual.
Therefore
\[
\varprojlim_{m\ge1} CH^i(F,j)/m
\simeq
\prod_l\varprojlim_n CH^i(F,j)/l^n
\]
is finite. Applying \autoref{lem:RS} with $\Sigma$ equal to the set of
all primes, we obtain a non-canonical decomposition
\[
CH^i(F,j)\simeq G\oplus D,
\]
where $G$ is finite and $D$ is divisible.

By the coefficient sequence with $\Q_l/\Z_l$-coefficients, there is a
surjection
\[
H^1_{\et}(F,\Q_l/\Z_l(i))
\to CH^i(F,j)[l^\infty].
\]
The structure theorem for divisible groups gives
\[
D\simeq U\oplus\bigoplus_l(\Q_l/\Z_l)^{I_l}
\]
for some $\Q$-vector space $U$. By \autoref{lem:coh-loc}, one has
$I_l=0$ for $l\neq p$ and $I_p\le[F:\Q_p]$. Consequently,
\[
CH^i(F,j)_{\mathrm{tor}}
\simeq G\oplus(\Q_p/\Z_p)^r
\]
for some $0\le r\le[F:\Q_p]$, and the quotient by the torsion
subgroup is uniquely divisible. This proves (1).

\smallskip
\noindent
(2) Assume that $2i-j=1$ and $\ch(F)=0$. For every prime
$l\neq p$, the coefficient exact sequence and the cycle class map give
an injection
\[
CH^i(F,j)/l^n
\hookrightarrow
H^1_{\et}(F,\Z/l^n(i)).
\]
Passing to the inverse limit over $n$, we obtain an injection
\[
\varprojlim_n CH^i(F,j)/l^n
\hookrightarrow
H^1_{\et}(F,\Z_l(i)).
\]
By local Tate duality,
\[
H^1_{\et}(F,\Z_l(i))
\simeq
H^1_{\et}(F,\Q_l/\Z_l(1-i))^\vee.
\]
By \autoref{lem:coh-loc}, this group is finite and is zero for all
but finitely many primes $l\neq p$. Therefore, for
$\Sigma=\{l\mid l\neq p\}$,
\[
\varprojlim_{m\in\mathbb N_\Sigma} CH^i(F,j)/m
\simeq
\prod_{l\neq p}\varprojlim_n CH^i(F,j)/l^n
\]
is finite. Applying \autoref{lem:RS} with this set $\Sigma$, we obtain
a non-canonical decomposition
\[
CH^i(F,j)\simeq G\oplus D,
\]
where $G$ is finite of order prime to $p$ and $D$ is
$p'$-divisible.

On the other hand, applying the coefficient exact sequence to
$CH^i(F,j+1)$ and passing to the direct limit gives, for every prime
$l$, a surjection
\[
H^0_{\et}(F,\Q_l/\Z_l(i))
\twoheadrightarrow
CH^i(F,j)[l^\infty].
\]
By \autoref{lem:coh-loc}, the group on the left is finite for every
$l$ and is zero for all but finitely many primes $l$. Hence
$CH^i(F,j)_{\mathrm{tor}}$ is finite.

For every prime $l\neq p$, the subgroup $D[l^\infty]$ is divisible,
because $D$ is $l$-divisible. Moreover,
\[
D[l^\infty]\subset CH^i(F,j)_{\mathrm{tor}},
\]
so $D[l^\infty]$ is finite. A finite divisible group is zero, and
therefore
\[
D[l^\infty]=0
\qquad(l\neq p).
\]
Thus $D$ is uniquely $p'$-divisible, and the above decomposition
shows that
\[
G=CH^i(F,j)\{p'\}.
\]
Consequently, the canonical short exact sequence
\[
0\to CH^i(F,j)\{p'\}
\to CH^i(F,j)
\to CH^i(F,j)/CH^i(F,j)\{p'\}\to0
\]
has uniquely $p'$-divisible quotient and splits non-canonically. This
proves (2).

\smallskip
\noindent
(3) Assume that $\ch(F)=p>0$ and $2i-j\in\{1,2\}$. The arguments in
(1) and (2), applied to primes $l\neq p$, give a non-canonical
decomposition
\[
CH^i(F,j)\simeq G\oplus D,
\]
where $G$ is finite of order prime to $p$ and $D$ is uniquely
$p'$-divisible. Since $[F:F^p]=p$ and $i\ge2$,
\autoref{prop:uniq-p-div}, applied with $s=2$ to $\Spec(F)$, shows that
$CH^i(F,j)$ is uniquely $p$-divisible. Hence $D$ is uniquely divisible and
$G=CH^i(F,j)_{\mathrm{tor}}$. This proves (3).

\smallskip
\noindent
(4) Suppose first that $i=j\ge3$. Then
$CH^i(F,j)\simeq K_i^M(F)$, and the assertion follows from Moore's
theorem \cite[Chapter~IX, Theorem~4.3]{FV02}, Merkur'ev's theorem
\cite[Chapter~IX, Theorem~4.7]{FV02}, and Sivitskii's theorem
\cite[Chapter~IX, Theorem~4.11]{FV02}.

Now let $l\neq\ch(F)$. Since $\cd_l(F)=2$,
\autoref{prop:l-div} shows that $CH^i(F,j)$ is uniquely
$l$-divisible if $2i-j\notin\{0,1,2,3\}$. If $2i-j=3$, then the
group is $l$-divisible and there is a surjection
\[
H^2(F,\Q_l/\Z_l(i))\twoheadrightarrow CH^i(F,j)[l^\infty].
\]
By \autoref{lem:coh-loc}, the source vanishes, so the group is
uniquely $l$-divisible. If $2i-j=0$, then
\[
\varprojlim_n CH^i(F,j)/l^n
\hookrightarrow H^0_{\et}(F,\Z_l(i))
\simeq H^2_{\et}(F,\Q_l/\Z_l(1-i))^\vee=0.
\]
This implies
$CH^i(F,j)/l^n=0$ for every $n$, and in particular
$CH^i(F,j)/l=0$.
Together with \autoref{prop:l-div}, this implies that
$CH^i(F,j)$ is uniquely $l$-divisible. In equal characteristic, since $[F:F^p]=p$ and $i\ge2$,
\autoref{prop:uniq-p-div}, applied with $s=2$ to $\Spec(F)$,
shows that $CH^i(F,j)$ is uniquely $p$-divisible.
Thus $CH^i(F,j)$ is uniquely divisible in all remaining cases.
\end{proof}

The following consequence of the finiteness theorem of
Gupta--Krishna--Rathore supplies the structure in the intermediate
range.

\begin{lem}\label{lem:local-structure}
Let $F$ be a local field with residue field $k$, and put
$p=\ch(k)$. Suppose that $i\ge2$. Then
$CH^{d+i}(X,j)\{p'\}$ is finite, and there is a canonical short exact
sequence
\[
0\to CH^{d+i}(X,j)\{p'\}
\to CH^{d+i}(X,j)
\to D\to0,
\]
where
\[
D:=CH^{d+i}(X,j)/CH^{d+i}(X,j)\{p'\}
\]
is uniquely $p'$-divisible. This sequence splits non-canonically. If
$\ch(F)>0$, then $D$ is uniquely divisible.
\end{lem}

\begin{proof}
Since $i\ge2$, we have $d+i\ge d+2$. Write
\[
(\Q/\Z)'=\bigoplus_{l\neq p}\Q_l/\Z_l.
\]
Since $X$ is proper, cohomology with compact support agrees with
ordinary \'etale cohomology. By \cite[Theorem~7.10]{GKR22}, the groups
$H^b_{\et}\left(X,(\Q/\Z)'(d+i)\right)$ 
are finite for every integer $b$. Moreover,
\cite[Corollary~7.11]{GKR22} gives a constant $M(i,b)$ such that
\[
\left|
H^b_{\et}\left(X,\Z/m(d+i)\right)
\right|
\le M(i,b)
\]
for every integer $m$ prime to $p$.
By \autoref{lem:cycl} and \eqref{eq:coeff}, there is an injection
\[
CH^{d+i}(X,j)/m
\hookrightarrow
H^{2(d+i)-j}_{\et}\left(X,\Z/m(d+i)\right).
\]
Thus the orders of $CH^{d+i}(X,j)/m$ are uniformly bounded as $m$
ranges over the integers prime to $p$.

Applying \eqref{eq:coeff} to $j+1$ and passing to the direct limit over
integers $m$ prime to $p$, we obtain a surjection
\[
H^{2(d+i)-j-1}_{\et}
\left(X,(\Q/\Z)'(d+i)\right)
\twoheadrightarrow
CH^{d+i}(X,j)\{p'\}.
\]
Hence $CH^{d+i}(X,j)\{p'\}$ is finite.

The uniform boundedness above implies that, for
$\Sigma=\{l\mid l\neq p\}$,
\[
\varprojlim_{m\in\mathbb N_\Sigma} CH^{d+i}(X,j)/m
\]
is finite. Applying \autoref{lem:RS} with this set $\Sigma$, we obtain
a non-canonical decomposition
\[
CH^{d+i}(X,j)\simeq G\oplus D,
\]
where $G$ is a finite group of order prime to $p$ and $D$ is
$p'$-divisible.

We show that $D$ is uniquely $p'$-divisible. Let $l\neq p$ be a
prime. Since $D$ is $l$-divisible, its $l$-primary torsion subgroup
$D[l^\infty]$ is divisible. On the other hand, since $D$ is a direct
summand of $CH^{d+i}(X,j)$, we have
$D[l^\infty]
\subset
CH^{d+i}(X,j)\{p'\}$.
The group on the right is finite. Hence $D[l^\infty]$ is both finite
and divisible, and therefore
$D[l^\infty]=0$.
Thus multiplication by $l$ on $D$ is both surjective and injective.
Since this holds for every prime $l\neq p$, the group $D$ is uniquely
$p'$-divisible.
Since $D$ has no prime-to-$p$ torsion, the above decomposition gives
$G=CH^{d+i}(X,j)\{p'\}$.
If $\ch(F)=p>0$, then $[F:F^p]=p$, and
\autoref{prop:uniq-p-div} shows that $CH^{d+i}(X,j)$ is uniquely
$p$-divisible. Since $D$ is a direct summand of $CH^{d+i}(X,j)$, it
is also uniquely $p$-divisible. Since $D$ is already uniquely
$p'$-divisible, it is uniquely divisible.
\end{proof}

We now apply \autoref{lem:CH-loc2}, \autoref{lem:local-structure} and the results of the previous section to the higher Chow group of smooth proper schemes over a local field $F$.

\begin{cor}\label{cor:CH-loc}
Let $F$ be a local field with residue field $k$, and put $p=\ch(k)$.
Let $X$ be a smooth proper and geometrically irreducible scheme over
$F$ of dimension $d$. Assume that $i\ge2$. Then the following holds.
\begin{enumerate}
\item If $2i-j\ge3$, then $CH^{d+i}(X,j)$ is uniquely divisible.
\item If $2i-j=2$, there is a canonical short exact sequence
\[
0\to A^{d+i}(X,j)
\to CH^{d+i}(X,j)
\xrightarrow{f_*}CH^i(F,j)\to0.
\]
This sequence splits non-canonically.
\item If $-2d<2i-j\le1$, then $CH^{d+i}(X,j)\{p'\}$ is finite and
there is a canonical short exact sequence
\[
0\to CH^{d+i}(X,j)\{p'\}
\to CH^{d+i}(X,j)
\to D\to0,
\]
where
\[
D:=CH^{d+i}(X,j)/CH^{d+i}(X,j)\{p'\}
\]
is uniquely $p'$-divisible. If $\ch(F)>0$, then $D$ is uniquely
divisible. This sequence splits non-canonically.
\item If $2i-j\le-2d$, then $CH^{d+i}(X,j)$ is uniquely divisible.
\end{enumerate}
\end{cor}
 
\begin{proof}
\smallskip
\noindent
(1) If $2i-j\ge4$, the assertion follows from
\autoref{prop:l-div} and, in equal characteristic,
\autoref{prop:uniq-p-div}, because $\cd_l(F)=2$ for every prime
$l\neq\ch(F)$ and $[F:F^p]=p$ if $\ch(F)>0$.

Assume that $2i-j=3$. By \autoref{prop:l-div},
$CH^{d+i}(X,j)$ is $l$-divisible for every prime
$l\neq\ch(F)$. Fix such a prime $l$. Since $\cd_l(F)=2$,
\autoref{thm:CH-div}\,(2) shows that $A^{d+i}(X,j)$ is uniquely
$l$-divisible. Moreover, \autoref{lem:CH-loc2} shows that
$CH^i(F,j)$ is uniquely $l$-divisible. Since
$2i-j=3\ge\cd_l(F)$, \autoref{thm:CH-div}\,(1) gives a short exact
sequence
\[
0\to A^{d+i}(X,j)[l]
\to CH^{d+i}(X,j)[l]
\xrightarrow{f_*}CH^i(F,j)[l]
\to0.
\]
Hence $CH^{d+i}(X,j)[l]=0$. Together with the $l$-divisibility
proved above, this shows that $CH^{d+i}(X,j)$ is uniquely
$l$-divisible. If $\ch(F)=p>0$, then
\autoref{prop:uniq-p-div} shows that it is uniquely $p$-divisible.
Thus $CH^{d+i}(X,j)$ is uniquely divisible.

\smallskip
\noindent
(2) Assume that $2i-j=2$. We first show that
$A^{d+i}(X,j)$ is divisible. Let $l\neq\ch(F)$ be a prime. Since
$\cd_l(F)=2$, \autoref{thm:CH-div}\,(2) shows that
$A^{d+i}(X,j)$ is $l$-divisible. If $\ch(F)=0$, this holds for every
prime $l$. Suppose that $\ch(F)=p>0$. Since $[F:F^p]=p$ and
$i\ge2$, \autoref{prop:uniq-p-div}, applied to both $X$ and
$\Spec(F)$, shows that $CH^{d+i}(X,j)$ and $CH^i(F,j)$ are uniquely
$p$-divisible. Let $a\in A^{d+i}(X,j)$ and write $a=pb$ with
$b\in CH^{d+i}(X,j)$. Since $pf_*(b)=f_*(a)=0$ and $CH^i(F,j)$ is
$p$-torsion free, we have $f_*(b)=0$. Hence
$b\in A^{d+i}(X,j)$, so $A^{d+i}(X,j)$ is $p$-divisible. Therefore
$A^{d+i}(X,j)$ is divisible in both cases.

We claim that $f_*$ is surjective. Let $CH^i(F,j)_{\mathrm{div}}$
denote the maximal divisible subgroup of $CH^i(F,j)$. By
\autoref{lem:CH-loc2}, there is a non-canonical decomposition
\[
CH^i(F,j)\simeq G\oplus CH^i(F,j)_{\mathrm{div}},
\]
where $G$ is finite. By \autoref{thm:CH-div}\,(1), the map $f_*$ is
surjective on the $l$-primary torsion subgroup for every prime
$l\neq\ch(F)$. If $\ch(F)=p>0$, then
\autoref{prop:uniq-p-div}, applied with $s=2$ to $\Spec(F)$,
shows that $CH^i(F,j)$ is uniquely $p$-divisible. In particular,
it has no nonzero $p$-primary torsion. Hence $G$ is contained in
the image of $f_*$. Choose a finite separable extension
$F'/F$ such that $X(F')\neq\emptyset$, put
$X_{F'}=X\otimes_FF'$, and let $\pi\colon X_{F'}\to X$ be the
projection. The structure morphism
$f'\colon X_{F'}\to\Spec(F')$ admits a section, so
$f'_*\colon CH^{d+i}(X_{F'},j)\to CH^i(F',j)$ is surjective. Let
$a\in CH^i(F,j)_{\mathrm{div}}$. Choose
$b\in CH^i(F,j)_{\mathrm{div}}$ such that $[F':F]b=a$, and choose
$z\in CH^{d+i}(X_{F'},j)$ such that
$f'_*(z)=\res_{F'/F}(b)$. From the commutative diagram
\[
\xymatrix{
CH^{d+i}(X_{F'},j)\ar@{->>}[r]^-{f'_*}\ar[d]_{\pi_*}
&CH^i(F',j)\ar[d]^{\Cor_{F'/F}}\\
CH^{d+i}(X,j)\ar[r]^-{f_*}&CH^i(F,j),
}
\]
we obtain
\[
f_*(\pi_*(z))
=\Cor_{F'/F}(f'_*(z))
=\Cor_{F'/F}\bigl(\res_{F'/F}(b)\bigr)
=a.
\]
Thus $f_*$ is surjective, and there is a short exact sequence
\[
0\to A^{d+i}(X,j)
\to CH^{d+i}(X,j)
\xrightarrow{f_*}CH^i(F,j)
\to0.
\]
Since $A^{d+i}(X,j)$ is divisible, this sequence splits
non-canonically.

\smallskip
\noindent
(3) This follows directly from \autoref{lem:local-structure}.

\smallskip
\noindent
(4) If $2i-j<-2d$, the assertion follows from
\autoref{prop:l-div} and, in equal characteristic,
\autoref{prop:uniq-p-div}.

Assume that $2i-j=-2d$. For every prime
$l\neq\ch(F)$, there is an injection
\[
CH^{d+i}(X,j)/l^n
\hookrightarrow H^0_{\et}(X,\Z/l^n(d+i))
\simeq H^0(F,\Z/l^n(d+i)).
\]
Passing to the inverse limit gives an injection
\[
\varprojlim_n CH^{d+i}(X,j)/l^n
\hookrightarrow H^0(F,\Z_l(d+i))
\simeq H^2(F,\Q_l/\Z_l(1-d-i))^\vee.
\]
The group on the right is zero by \autoref{lem:coh-loc}. Suppose
first that $\ch(F)=0$. Then
\[
\varprojlim_{m\ge1} CH^{d+i}(X,j)/m=0,
\]
so \autoref{lem:RS}, applied with $\Sigma$ equal to the set of all
primes, gives a non-canonical decomposition
\[
CH^{d+i}(X,j)\simeq G\oplus D,
\]
where $G$ is finite and $D$ is divisible. By
\autoref{prop:l-div}, applied to every prime $l$, the group
$CH^{d+i}(X,j)$ is torsion free. Hence $G=0$, and $D$ is uniquely
divisible.

Suppose next that $\ch(F)=p>0$, and put
$\Sigma=\{l\mid l\neq p\}$. Then
\[
\varprojlim_{m\in\mathbb N_\Sigma} CH^{d+i}(X,j)/m=0.
\]
By \autoref{lem:RS}, there is a non-canonical decomposition
\[
CH^{d+i}(X,j)\simeq G\oplus D,
\]
where $G$ is finite of order prime to $p$ and $D$ is
$p'$-divisible. Moreover, \autoref{prop:uniq-p-div}, applied with
$s=2$, shows that $CH^{d+i}(X,j)$ is uniquely $p$-divisible. Thus
$D$ is uniquely $p$-divisible. Since \autoref{prop:l-div} gives
prime-to-$p$ torsion-freeness, the whole group is torsion free.
Therefore $G=0$, and $D$ is uniquely divisible.
\end{proof}

\begin{cor}\label{cor:A-loc}
Under the assumptions of \autoref{cor:CH-loc}, the structure of
$A^{d+i}(X,j)$ is as follows.
\begin{enumerate}
\item If $2i-j\ge2$, then $A^{d+i}(X,j)$ is uniquely divisible.
\item Suppose that $2i-j=1$. If $\ch(F)=0$, let
$A^{d+i}(X,j)_{\mathrm{div}}$ denote the maximal divisible subgroup
of $A^{d+i}(X,j)$. Then there is a canonical short exact sequence
\[
0\to A^{d+i}(X,j)_{\mathrm{div}}
\to A^{d+i}(X,j)
\to A^{d+i}(X,j)/A^{d+i}(X,j)_{\mathrm{div}}
\to0,
\]
where the quotient is finite. This sequence splits non-canonically.
If $\ch(F)>0$, then $A^{d+i}(X,j)\{p'\}$ is finite and there is a
canonical short exact sequence
\[
0\to A^{d+i}(X,j)\{p'\}
\to A^{d+i}(X,j)
\to D_A\to0,
\]
where
\[
D_A:=A^{d+i}(X,j)/A^{d+i}(X,j)\{p'\}
\]
is uniquely divisible. This sequence splits non-canonically.
\item If $-2d<2i-j\le0$, then $A^{d+i}(X,j)\{p'\}$ is finite and
there is a canonical short exact sequence
\[
0\to A^{d+i}(X,j)\{p'\}
\to A^{d+i}(X,j)
\to D_A\to0,
\]
where
\[
D_A:=A^{d+i}(X,j)/A^{d+i}(X,j)\{p'\}
\]
is uniquely $p'$-divisible. If $\ch(F)>0$, then $D_A$ is uniquely
divisible. This sequence splits non-canonically.
\item If $2i-j\le-2d$, then $A^{d+i}(X,j)$ is uniquely divisible.
\end{enumerate}
\end{cor}

\begin{proof}
\smallskip
\noindent
(1) If $2i-j\ge3$, then \autoref{thm:CH-div}\,(2) shows that
$A^{d+i}(X,j)$ is uniquely $l$-divisible for every prime
$l\neq\ch(F)$. If $\ch(F)=p>0$, the same kernel argument as in the
proof of \autoref{cor:CH-loc}, together with
\autoref{prop:uniq-p-div} and \autoref{lem:CH-loc2}, shows that it is
uniquely $p$-divisible. Thus $A^{d+i}(X,j)$ is uniquely divisible.
If $2i-j=2$, the proof of \autoref{cor:CH-loc}\,(2) shows that
$A^{d+i}(X,j)$ is divisible.
It remains to prove that $A^{d+i}(X,j)$ is torsion
free. For every prime $l\neq\ch(F)$, \autoref{thm:CH-div}\,(3) gives
a surjection
\[
H^2_{\et}\left(
F,H^{2d-1}_{\et}\left(X_{F^{\sep}},\Q_l/\Z_l(d+i)\right)
\right)
\to A^{d+i}(X,j)[l^\infty].
\]
The source vanishes by \autoref{lem:E_2-term} below. If $\ch(F)=p>0$,
then $A^{d+i}(X,j)$ is also $p$-torsion free because it is a subgroup
of the uniquely $p$-divisible group $CH^{d+i}(X,j)$. Hence
$A^{d+i}(X,j)$ is uniquely divisible.

\smallskip
\noindent
(2) Assume that $2i-j=1$. Fix a prime $l\neq\ch(F)$, and put
\[
\widehat{A}_l:=\varprojlim_n A^{d+i}(X,j)/l^n.
\]
We claim that $\widehat{A}_l$ is finite and zero for all but finitely
many primes $l$. Applying \eqref{eq:coeff} with $r=i$, $m=l^n$, and
$j$ replaced by $j+1$, and passing to the direct limit over $n$, we
obtain a surjection
\[
H^0_{\et}(F,\Q_l/\Z_l(i))
\to CH^i(F,j)[l^\infty].
\]
By \autoref{lem:coh-loc}, the group on the left is finite. Hence
$CH^i(F,j)[l^\infty]$ has bounded exponent.

We use the following elementary observation. For an exact sequence
\[
0\to A\to C\to B,
\]
if $B[l^\infty]$ has bounded exponent, then the natural map
\[
\varprojlim_n A/l^n\to\varprojlim_n C/l^n
\]
is injective. Indeed, suppose that $l^NB[l^\infty]=0$. If
$a\in A$ maps to zero in $C/l^n$, write $a=l^nc$ with $c\in C$.
The image of $c$ in $B$ belongs to $B[l^n]$, so $l^Nc\in A$ and
$a=l^{n-N}(l^Nc)$ for $n\ge N$.

Applying this observation to
\[
0\to A^{d+i}(X,j)
\to CH^{d+i}(X,j)
\xrightarrow{f_*}\operatorname{Im}(f_*)
\to0,
\]
we obtain an injection
\[
\widehat{A}_l
\hookrightarrow\varprojlim_n CH^{d+i}(X,j)/l^n.
\]
Taking the inverse limit of the finite-coefficient cycle class maps
gives a commutative diagram
\[
\xymatrix{
\widehat{A}_l\ar@{^{(}->}[r]\ar[d]
&\displaystyle\varprojlim_n CH^{d+i}(X,j)/l^n
\ar@{^{(}->}[r]\ar[d]
&H^{2d+1}_{\et}(X,\Z_l(d+i))\ar[d]^{\Tr_f}\\
0\ar[r]
&\displaystyle\varprojlim_n CH^i(F,j)/l^n
\ar@{^{(}->}[r]
&H^1_{\et}(F,\Z_l(i)).
}
\]
Consequently, $\widehat{A}_l\hookrightarrow\Ker(\Tr_f)$. The
Hochschild--Serre spectral sequence in total degree $2d+1$ gives
\[
\Ker(\Tr_f)\simeq E_\infty^{2,2d-1}\simeq E_3^{2,2d-1},
\]
and there is a natural surjection
\[
H^2_{\et}\left(
F,H^{2d-1}_{\et}\left(X_{F^{\sep}},\Z_l(d+i)\right)
\right)
\to E_3^{2,2d-1}.
\]
By \autoref{lem:E_2-term} below, $\widehat{A}_l$ is finite and zero for all
but finitely many primes $l$.

Suppose first that $\ch(F)=0$. Since the preceding argument applies
to every prime $l$, we have
\[
\varprojlim_{m\ge1} A^{d+i}(X,j)/m
\simeq
\prod_l \widehat{A}_l,
\]
and this group is finite. Applying \autoref{lem:RS} with $\Sigma$
equal to the set of all primes, we obtain a non-canonical
decomposition
\[
A^{d+i}(X,j)\simeq G_A\oplus D_A,
\]
where $G_A$ is finite and $D_A$ is divisible. Therefore
$A^{d+i}(X,j)/A^{d+i}(X,j)_{\mathrm{div}}$ is finite, and the
canonical short exact sequence in the statement splits
non-canonically.

Suppose next that $\ch(F)=p>0$, and put
$\Sigma=\{l\mid l\neq p\}$. The preceding argument gives
\[
\varprojlim_{m\in\mathbb N_\Sigma} A^{d+i}(X,j)/m
\simeq
\prod_{l\neq p}\widehat{A}_l,
\]
and this group is finite. Applying \autoref{lem:RS} with this set
$\Sigma$, we obtain a non-canonical decomposition
\[
A^{d+i}(X,j)\simeq G_A\oplus D_A,
\]
where $G_A$ is finite of order prime to $p$ and $D_A$ is
$p'$-divisible. Since
\[
A^{d+i}(X,j)\{p'\}\subset CH^{d+i}(X,j)\{p'\},
\]
and the group on the right is finite by
\autoref{lem:local-structure}, the prime-to-$p$ torsion subgroup of
$D_A$ is zero. Hence $D_A$ is uniquely $p'$-divisible. Moreover, applying \autoref{prop:uniq-p-div} with $s=2$ to
both $X$ and $\Spec(F)$, we see that $CH^{d+i}(X,j)$ and
$CH^i(F,j)$ are uniquely $p$-divisible. The usual kernel argument
shows that $A^{d+i}(X,j)$, and hence $D_A$, is uniquely
$p$-divisible. Thus
$D_A$ is uniquely divisible and
$G_A=A^{d+i}(X,j)\{p'\}$. This gives the canonical short exact
sequence in the statement.

\smallskip
\noindent
(3) Assume that $-2d<2i-j\le0$. By \autoref{lem:CH-loc2}, the group
$CH^i(F,j)$ is uniquely divisible. As in the case $2i-j=2$, the map
$f_*$ is surjective, and there is a short exact sequence
\[
0\to A^{d+i}(X,j)
\to CH^{d+i}(X,j)
\xrightarrow{f_*}CH^i(F,j)
\to0.
\]
For every integer $m\ge1$, the natural maps
\[
A^{d+i}(X,j)/m\xrightarrow{\simeq}CH^{d+i}(X,j)/m,
\qquad
A^{d+i}(X,j)[m]\xrightarrow{\simeq}CH^{d+i}(X,j)[m]
\]
are isomorphisms. Hence $A^{d+i}(X,j)\{p'\}$ is finite. Moreover,
the proof of \autoref{lem:local-structure} shows that the orders of
$CH^{d+i}(X,j)/m$, and therefore those of $A^{d+i}(X,j)/m$, are
uniformly bounded as $m$ ranges over the integers prime to $p$.
Thus, for $\Sigma=\{l\mid l\neq p\}$,
\[
\varprojlim_{m\in\mathbb N_\Sigma} A^{d+i}(X,j)/m
\]
is finite. Applying \autoref{lem:RS} with this set $\Sigma$, we obtain
a non-canonical decomposition
\[
A^{d+i}(X,j)\simeq G_A\oplus D_A,
\]
where $G_A$ is finite of order prime to $p$ and $D_A$ is
$p'$-divisible. The prime-to-$p$ torsion subgroup of $D_A$ is zero,
so $D_A$ is uniquely $p'$-divisible and
$G_A=A^{d+i}(X,j)\{p'\}$. If $\ch(F)=p>0$, then
\autoref{prop:uniq-p-div}, applied with $s=2$ to both $X$ and
$\Spec(F)$, shows that $CH^{d+i}(X,j)$ and $CH^i(F,j)$ are
uniquely $p$-divisible. The usual kernel argument then shows that
$A^{d+i}(X,j)$ is uniquely $p$-divisible, so $D_A$ is uniquely
divisible. This gives the canonical short exact sequence in the
statement.

\smallskip
\noindent
(4) If $2i-j<-2d$, the assertion follows from
\autoref{prop:l-div} and, in equal characteristic,
\autoref{prop:uniq-p-div}. Assume that $2i-j=-2d$. By
\autoref{cor:CH-loc}\,(4), $CH^{d+i}(X,j)$ is uniquely divisible.
Let $a\in A^{d+i}(X,j)$ and $n\ge1$. There is a unique
$b\in CH^{d+i}(X,j)$ such that $a=nb$. Since
$nf_*(b)=f_*(a)=0$ and $CH^i(F,j)$ is uniquely divisible, we have
$f_*(b)=0$. Hence $b\in A^{d+i}(X,j)$. Thus multiplication by $n$
on $A^{d+i}(X,j)$ is bijective, and the group is uniquely divisible.
\end{proof}

\begin{lem}\label{lem:E_2-term}
    Under the assumptions of \autoref{cor:CH-loc},
    the following statements hold.
    \begin{enumerate}
        \item 
    $
    H_\et^{2}(F,H_\et^{2d-1}(X_{F^\sep},\Q_l/\Z_l(d+i)))=0
    $
    for every prime $l\neq \ch(F)$.
    \item
    $
    H_\et^{2}(F,H_\et^{2d-1}(X_{F^\sep},\Z_l(d+i)))
    $
    is finite for every prime $l\neq\ch(F)$ and it is zero for all but finitely many primes $l$. 
    \end{enumerate}
\end{lem}

\begin{proof}
    Fix a prime $l\neq\ch(F)$. 
    Put
    $\overline X=X_{F^{\sep}} = X\otimes_F F^{\sep}$ 
    and let $\Pi_l := \pi_1^{\mathrm{ab}}(\overline X)^{(l)}$ 
    denote the maximal pro-$l$ quotient of the abelian fundamental group $\pi_1^{\mathrm{ab}}(\overline X)$.

    By geometric Poincar\'e duality and local Tate duality, we have
    \begin{align*}
        H_\et^{2}(F,H_\et^{2d-1}(X_{F^\sep}, \Q_l/\Z_l(d+i)))
        &\simeq 
        H_\et ^2(F,\Pi_l\otimes \Q_l/\Z_l(i))
         \\
         &\simeq 
          H^0(F, (\Pi_l)^\vee\otimes \Z_l(1-i))^\vee\\
          &\simeq 
          (\Pi_l \otimes \Q_l/\Z_l(i-1))_{G_F}.
    \end{align*}
    Similarly, we have
    $$
    H_\et^{2}(F,H_\et^{2d-1}(X_{F^\sep}, \Z_l(d+i)))
    \simeq 
    (\Pi_l\otimes \Z_l(i-1))_{G_F}.
    $$
    For any finite extension $F'/F$ and a Galois module $M$, the corestriction map
    $
    \mathrm{cor}_{F'/F}: M_{G_{F'}}\to M_{G_F}
    $
    is surjective.
    Therefore, in order to show the statement, we may replace $F$ by any finite extension $F'$, and hence we may assume that $X(F)\neq \emptyset$. Put $A = \mathrm{Alb}_X$, the Albanese variety of $X$. 
    
    By \cite[Lemma 5]{KL81} (cf.~\cite[Lemma 4.1]{Yos03}), we have an exact sequence of $G_{F}$-modules
     \begin{equation}\label{ex:pi1}
         0\to C\to \pi_1^{\mathrm{ab}}(\overline X)\to T^\et(A)\to 0.
    \end{equation}
     Here $C$ is finite and $T^\et(A)$ is the maximal \'etale quotient of the Tate module of $A$.
     Taking maximal pro-$l$ quotients we obtain an exact sequence 
     \begin{equation}\label{ex:pi1l}
        0 \to C_l \to \Pi_l \to T_l(A) \to 0,    
     \end{equation}
    where $C_l$ is the image of $C$ in $\Pi_l$ and $T_l(A) = \varprojlim_n A[l^n]$ is the $l$-adic Tate module.
    Since $C_l$ is finite, $C_l\otimes_{\Z_l}\Q_l/\Z_l = 0$. 
    Since $T_l(A)$ is a finite free $\Z_l$-module, we have 
    \[
    \Pi_l\otimes_{\Z_l}\Q_l/\Z_l(i-1) \xrightarrow{\simeq} T_l(A)\otimes_{\Z_l}\Q_l/\Z_l(i-1).
    \]
    Consequently,
    \begin{equation}\label{eq:pi1-Tate-div}
    \left(\Pi_l\otimes_{\Z_l}\Q_l/\Z_l(i-1)\right)_{G_F}
    \simeq
    \left(T_l(A)\otimes_{\Z_l}\Q_l/\Z_l(i-1)\right)_{G_F}.
    \end{equation}
    On the other hand, \eqref{ex:pi1l} gives an exact sequence 
    \begin{equation}\label{eq:pi1-Tate-int}
	\left(C_l\otimes_{\Z_l}\Z_l(i-1)\right)_{G_F}
	\to \left(\Pi_l\otimes_{\Z_l}\Z_l(i-1)\right)_{G_F}
	\to \left(T_l(A)\otimes_{\Z_l}\Z_l(i-1)\right)_{G_F}
	\to 0.
    \end{equation}
    The first term is finite and is zero for all but finitely many primes $l$. 
    It is therefore enough to prove the following:
    \begin{enumerate}[label=(\alph*)]
    \item
    $\left(T_l(A)\otimes_{\Z_l}\Q_l/\Z_l(i-1)
    \right)_{G_F}=0$,
    \item 
    $\left(T_l(A)\otimes_{\Z_l}\Z_l(i-1)\right)_{G_F}$ 
    is finite and is zero for all but finitely many primes $l$.
    \end{enumerate}

     By the semistable reduction theorem, there is a finite separable extension $F'/F$ such that $\mathrm{Alb}_{X_{F'}}$ has semistable reduction.
     By replacing $F$ by $F'$, we may assume that $A$ has semistable reduction.
     Replacing $F$ again by a finite separable extension if necessary, we may further assume that $A$ has \emph{split} semistable reduction.
     After a further finite extension, we may also assume that
    the torus and the lattice occurring in the rigid uniformization below
    are split; more precisely, the torus is split and the $G_F$-action on
    the lattice is trivial. All these finite extensions are chosen independently of $l$.
     By the non-archimedean rigid uniformization theorem \cite[Section~1]{BX96}, replacing $F$ by a further finite extension if necessary, we obtain the following data:
     \begin{enumerate}
         \item[(i)]
         a semi-abelian variety $S$ over $F$ fitting into the following exact sequence
         $$
         0\to\Gm^r \to S\to B\to 0 
         $$
         where $B$ is an abelian variety with good reduction.

         \item[(ii)]
         a closed immersion of rigid analytic groups $N^\mathrm{an}\to S^\mathrm{an}$ where $N$ is a group scheme which is isomorphic to $\Z^r$.
         
         \item[(iii)]
         a faithfully flat morphism of rigid analytic groups $S^\mathrm{an}\to A^\mathrm{an}$ which induces an isomorphism $S^\mathrm{an}/N^\mathrm{an}\simeq A^\mathrm{an}$. 
     \end{enumerate}
     This gives exact sequences
     $$
     0\to T_l(\Gm^r)\to T_l(S)\to T_l(B)\to 0,\quad 0\to T_l(S)\to T_l(A)\to N\otimes_\Z \Z_l \to 0.
     $$
     Since $T_l(\Gm^r)=\Z_l(1)^r$ and $\varprojlim_n N/l^n= N\otimes_\Z \Z_l  \simeq  \Z_l^r$, for any $l\neq \ch(F)$, from \autoref{lem:coh-loc} we obtain that
     \begin{align*}
         &(T_l(\Gm^r)\otimes \Q_l/\Z_l(i-1))_{G_F}\simeq H^2(F,\Q_l/\Z_l(i+1))^r =0,\\
         &(( N\otimes_\Z \Z_l )\otimes \Q_l/\Z_l(i-1))_{G_F}\simeq H^2(F,\Q_l/\Z_l(i))^r=0,
     \end{align*}
     and that
     \begin{align*}
         &(T_l(\Gm^r)\otimes \Z_l(i-1))_{G_F}\simeq (H^0(F,\Q_l/\Z_l(-i))^\vee)^r,\\
         &\left( (N\otimes_\Z \Z_l)\otimes \Z_l(i-1)\right)_{G_F}\simeq (H^0(F, \Q_l/\Z_l(1-i))^\vee)^r 
     \end{align*}
     are finite and trivial for all but finitely many primes $l$.
     Hence, it is enough to show that 
     \begin{enumerate}[label=(\alph*')]
         \item 
         $(T_l(B)\otimes \Q_l/\Z_l(i-1))_{G_F}=0$
         \item
         $(T_l(B)\otimes \Z_l(i-1))_{G_F}$ is finite and is zero for all but finitely many primes $l$.
     \end{enumerate}

    We first treat the case $l\neq p$. 
    Since $B$ has good reduction, the action of $G_F$ on $T_l(B)$ is unramified and factors through $G_k$.
    Let $B_k$ denote the reduction of the N\'eron model of $B$. 
    By smooth proper base change,  
    $T_l(B) \simeq T_l(B_k)$ as $G_k$-modules. Therefore, 
    we have
    $$
     (T_l(B)\otimes \Lambda(i-1))_{G_F}\simeq (T_l(B_k)\otimes \Lambda(i-1))_{G_k},
    $$
    where $\Lambda=\Z_l$ or $\Q_l/\Z_l$. 
    Let $q = \#k$ and let $\operatorname{Frob}_q$ denote the arithmetic Frobenius. 
    Thus $\operatorname{Frob}_q$ acts on $\Z_l(1)$ as multiplication by $q$. 
    Put $V_l(B_k) = T_l(B_k)\otimes_{\Z_l}\Q_l$. 
    By the Weil conjectures, every complex conjugate of every eigenvalue of $\operatorname{Frob}_q$ on $V_l(B_k)$
    has absolute value $q^{1/2}$. 
    Hence every eigenvalue of $\operatorname{Frob}_q$ on $V_l(B_k)(i-1)$ has absolute value $q^{i-1/2}$. 
    Since $i\ge2$, the number $1$ is not an eigenvalue. Therefore,
    \[
    \operatorname{Frob}_q-1\colon V_l(B_k)(i-1)\to V_l(B_k)(i-1)
    \]
    is an isomorphism. 
    We have the following commutative diagram
    $$
    \xymatrix{
    0\ar[r] 
    &T_l(B_k)\otimes \Z_l(i-1)\ar[r]\ar[d]^-{\operatorname{Frob}_q-1} 
    & V_l(B_k)(i-1)\ar[r]\ar[d]_{\simeq}^-{\operatorname{Frob}_q-1} 
    & T_l(B_k)\otimes \Q_l/\Z_l(i-1) \ar[r]\ar[d]^-{\operatorname{Frob}_q-1}
    &0\\
    0\ar[r]
    & T_l(B_k)\otimes \Z_l(i-1)\ar[r]
    & V_l(B_k)(i-1)\ar[r] 
    & T_l(B_k)\otimes \Q_l/\Z_l(i-1) \ar[r]&0
    }
    $$
    which proves assertion (a').
    Since the left vertical map is injective and $T_l(B_k)\otimes \Z_l(i-1)$ is a free $\Z_l$-module of finite rank, the coinvariant $(T_l(B_k)\otimes \Z_l(i-1))_{G_k}$ is finite. 
    To prove the assertion for almost all $l$, 
    we denote by  
    \[
    P(t) = \det(1-t\operatorname{Frob}_q | V_l(B_k))
    \]
    the characteristic polynomial of $\operatorname{Frob}_q$ on $V_l(B_k)$.
    By the Weil conjectures, $P(t)\in \Z[t]$ and is independent of $l$. 
    Since $\operatorname{Frob}_q$ acts on $\Q_l(i-1)$ as
    multiplication by $q^{i-1}$, its action on $V_l(B_k)(i-1)$ is
    $q^{i-1}\operatorname{Frob}_q$. Hence
    \[
    \det\left(\operatorname{Frob}_q-1\mathrel{\big|}V_l(B_k)(i-1)\right)=P(q^{i-1}),
    \]
    because $\dim_{\Q_l}V_l(B_k)=2\dim B$ is even.
    Consequently,     
    if $l\nmid P(q^{i-1})$, then $\operatorname{Frob}_q-1$ is bijective on $T_l(B_k)\otimes \Z_l(i-1)$, which proves assertion (b').
    
    Lastly, we consider the case $\ch(F)=0$ and $l=p$. 
    In this case $T_p(B)$ is the usual $p$-adic Tate module, since the
    ground field has characteristic zero.
    Put $V_p(B)=T_p(B)\otimes_{\Z_p}\Q_p$.
    The Weil pairing gives a $G_F$-equivariant isomorphism
    \[
        \Hom_{\Z_p}\left(T_p(B),\Z_p\right) \simeq T_p(B^\vee)(-1).
    \]
    Therefore,
    \[
    \left(\left(T_p(B)\otimes_{\Z_p}\Q_p/\Z_p(i-1)\right)_{G_F}\right)^\vee 
    \simeq \left(T_p(B^\vee)\otimes_{\Z_p}\Z_p(-i) \right)^{G_F} 
    \subset V_p(B^\vee)(-i)^{G_F}.
    \]
    We use the convention that $\Q_p(1)$ has Hodge--Tate weight $-1$.
    Faltings' Hodge--Tate decomposition (\cite[Part~III, Theorem~4.1]{Fal88}) gives
    \[
    H^1_{\et}\left(B_{\overline F},\Q_p\right)\otimes_{\Q_p}\C_p
\simeq
H^1\left(B,\mathcal O_B\right)\otimes_F\C_p
\oplus
H^0\left(B,\Omega^1_{B/F}\right)\otimes_F\C_p(-1).
    \]
    Thus
    $H^1_{\et}\left(B_{\overline F},\Q_p\right)$ 
    has Hodge--Tate weights $0$ and $1$, each with multiplicity
    $\dim(B)$. Since
    $H^1_{\et}\left(B_{\overline F},\Q_p\right)
    \simeq
    V_p(B)^\vee$,
    the covariant Tate module $V_p(B)$ has Hodge--Tate weights $0$ and $-1$.
    The same applies to $B^\vee$. Consequently,
    $V_p(B^\vee)(-i)$ has Hodge--Tate weights $i$ and $i-1$.
    Since $i\ge2$, neither of these weights is zero. Hence
    \[
    V_p(B^\vee)(-i)^{G_F}=0.
    \]
    Indeed, a nonzero invariant vector would define a copy of the
    trivial representation $\Q_p$, whose Hodge--Tate weight is zero.
    Consequently
    \[
    \left(T_p(B)\otimes_{\Z_p}\Q_p/\Z_p(i-1)\right)_{G_F}=0.
    \]
    This proves (a').

    Finally, if $(T_p(B)\otimes \Z_p(i-1))_{G_F}$ is infinite, then 
    it has positive $\Z_p$-rank, and hence 
    there is a nonzero $\Z_p$-linear homomorphism $(T_p(B)\otimes \Z_p(i-1))_{G_F} \to \Z_p$. 
    Composing with the quotient map gives a non-trivial  $G_F$-morphism 
    $T_p(B)\otimes \Z_p(i-1)\to \Z_p$. 
    After tensoring with $\Q_p$, 
    this gives a non-trivial $G_F$-morphism
    $$
    V_p(B)(i-1)\to \Q_p.
    $$
    By the Weil pairing 
    \[
    \Hom_{\Q_p}(V_p(B)(i-1),\Q_p) \simeq V_p(B^\vee)(-i). 
    \]
    Thus the preceding homomorphism gives a nonzero element of $V_p(B^{\vee})(-i)^{G_F}$, 
    contradicting the Hodge--Tate weight argument above.
    Therefore $(T_p(B)\otimes \Z_p(i-1))_{G_F}$ is finite, which completes the proof.
\end{proof}

\subsection*{Global fields}

We first determine the structure of the higher Chow groups of a
global field.

\begin{lem}\label{lem:global}
Let $F$ be a global field such that either $\ch(F)>0$ or $F$ is a
totally imaginary number field. Assume that $i\ge2$. Then the
following hold.
\begin{enumerate}
\item If $2i-j=2$, then $CH^i(F,j)$ is an infinite torsion group.
\item If $2i-j=1$, then $CH^i(F,j)_{\mathrm{tor}}$ is finite and
there is a canonical short exact sequence
\[
0\to CH^i(F,j)_{\mathrm{tor}}
\to CH^i(F,j)
\to CH^i(F,j)/CH^i(F,j)_{\mathrm{tor}}\to0,
\]
where $CH^i(F,j)/CH^i(F,j)_{\mathrm{tor}}$ is a free abelian group
of rank
\[
r=
\begin{cases}
r_2, & \mbox{if $F$ is a totally imaginary number field},\\
0, & \mbox{if $\ch(F)>0$}.
\end{cases}
\]
This sequence splits non-canonically. Here $r_2$ denotes the number
of complex places of $F$.
\item In all other cases, $CH^i(F,j)=0$.
\end{enumerate}
\end{lem}

\begin{proof}
Let $U$ be a regular one-dimensional global model of $F$. If $F$ is
a number field, we take $U=\Spec(\mathcal O_F)$. If $\ch(F)>0$, we
take $U$ to be the smooth projective curve over a finite field with
function field $F$.

\smallskip
\noindent
(1) Assume that $2i-j=2$. If $i=2$, then
$CH^2(F,2)\simeq K_2^M(F)$ is an infinite torsion group by
\cite[Theorem~2.1]{BT73}. We may therefore assume that $i\ge3$.

The localization sequence contains
\[
CH^i(F,2i-2)\to
\bigoplus_{x\in U_0}CH^{i-1}(\kappa(x),2i-3)
\to CH^i(U,2i-3).
\]
The group $CH^i(U,2i-3)$ is finitely generated by the known case of
the motivic Bass conjecture (\cite[Theorem~40]{Kah05}). Moreover,
$CH^i(U,2i-3)\otimes_\Z\Q\simeq K_{2i-3}(U)^{(i)}_\Q=0$ by Borel's
theorem in characteristic zero
(\cite[Theorem~7]{Wei05}) and by Harder's theorem in positive
characteristic. Hence $CH^i(U,2i-3)$ is finite
(\cite[Theorem~48]{Wei05}).

On the other hand, \autoref{lem:fin} gives
\[
CH^{i-1}(\kappa(x),2i-3)
\simeq \Z/(\#\kappa(x)^{\,i-1}-1)
\]
for every $x\in U_0$. Since $U$ has infinitely many closed points,
the direct sum in the middle is infinite. Its kernel under the
second map is therefore infinite. Consequently,
$CH^i(F,2i-2)$ is infinite.

Furthermore,
$CH^i(F,2i-2)\otimes_\Z\Q\simeq K_{2i-2}(F)^{(i)}_\Q=0$ by Borel's
theorem in characteristic zero and by Harder's theorem together with
the localization sequence in positive characteristic. Thus
$CH^i(F,2i-2)$ is an infinite torsion group.

\smallskip
\noindent
(2) Assume that $2i-j=1$, so that $j=2i-1$. The localization
sequence gives
\[
CH^i(U,2i-1)\to CH^i(F,2i-1)\to
\bigoplus_{x\in U_0}CH^{i-1}(\kappa(x),2i-2).
\]
Since $2(i-1)-(2i-2)=0$, the right-hand term vanishes by
\autoref{lem:fin}. Hence $CH^i(U,2i-1)\to CH^i(F,2i-1)$ is
surjective. The known case of the motivic Bass conjecture therefore
shows that $CH^i(F,2i-1)$ is finitely generated.

If $F$ is a totally imaginary number field, Borel's theorem gives
$\operatorname{rank}_\Z CH^i(F,2i-1)=r_2$. If $\ch(F)>0$, Harder's
theorem gives $CH^i(F,2i-1)\otimes_\Z\Q=0$. Consequently,
$CH^i(F,2i-1)_{\mathrm{tor}}$ is finite and
$CH^i(F,2i-1)/CH^i(F,2i-1)_{\mathrm{tor}}$ is a free abelian group
of rank $r$. Thus there is a canonical short exact sequence
\[
0\to CH^i(F,2i-1)_{\mathrm{tor}}
\to CH^i(F,2i-1)
\to CH^i(F,2i-1)/CH^i(F,2i-1)_{\mathrm{tor}}\to0.
\]
Since the quotient is free abelian, it is projective, and the
sequence splits non-canonically.

\smallskip
\noindent
(3) It remains to consider the cases $2i-j\ge3$ and $2i-j\le0$.
Suppose first that $2i-j\ge3$. If $j<i$, then $CH^i(F,j)=0$ for
dimensional reasons. If $j=i$, then $i\ge3$ and
$CH^i(F,i)\simeq K_i^M(F)=0$ by \cite[Theorem~2.1]{BT73}. We may
therefore assume that $j>i$.

The localization sequence gives
\begin{equation}\label{eq:loc-global}
CH^i(U,j)\to CH^i(F,j)\to
\bigoplus_{x\in U_0}CH^{i-1}(\kappa(x),j-1).
\end{equation}
Since $2(i-1)-(j-1)=2i-j-1\ge2$, \autoref{lem:fin} shows that
$CH^{i-1}(\kappa(x),j-1)=0$ for every $x\in U_0$. Hence
$CH^i(U,j)\to CH^i(F,j)$ is surjective.

By the known case of the motivic Bass conjecture
(\cite[Theorem~40]{Kah05}), the group $CH^i(U,j)$ is
finitely generated. Moreover, \cite[Theorem~3.1]{Lev94} gives
$CH^i(U,j)\otimes_\Z\Q\simeq K_j(U)^{(i)}_\Q$.

Suppose that $\ch(F)=0$. If $j$ is even, Borel's theorem gives
$K_j(U)_\Q=0$. If $j$ is odd, write $j=2a-1$. Then
$K_j(U)_\Q=K_j(U)^{(a)}_\Q$. Since $2i-j\ge3$, we have $i\neq a$,
and hence $K_j(U)^{(i)}_\Q=0$. If $\ch(F)>0$, then Harder's theorem
gives $K_j(U)^{(i)}_\Q=0$. Thus $CH^i(U,j)$ is finite.

For every prime $l\neq\ch(F)$, \autoref{prop:l-div}, applied with
$s=2$ and $d=0$, shows that $CH^i(F,j)$ is $l$-divisible. If
$\ch(F)=p>0$, then $[F:F^p]=p$, and
\autoref{prop:uniq-p-div}, applied with $s=2$ and $d=0$, shows that
$CH^i(F,j)$ is uniquely $p$-divisible. Since $CH^i(F,j)$ is a
quotient of the finite group $CH^i(U,j)$, it follows that
$CH^i(F,j)=0$.

Finally, suppose that $2i-j\le0$. The rational comparison with
Quillen $K$-theory, together with Borel's theorem in characteristic
zero and Harder's theorem in positive characteristic, gives
$CH^i(F,j)\otimes_\Z\Q=0$. Thus $CH^i(F,j)$ is a torsion group. Let
$l\neq\ch(F)$ be a prime. Since $2i-j-1<0$,
\autoref{lem:cycl} gives $CH^i(F,j+1;\Z/l)=0$. The coefficient exact
sequence therefore gives $CH^i(F,j)[l]=0$. If $\ch(F)=p>0$, then
\autoref{prop:uniq-p-div} also gives $CH^i(F,j)[p]=0$. Hence
$CH^i(F,j)$ is torsion free, and therefore zero.
\end{proof}

We now apply \autoref{lem:global} and the results of the previous
section to the higher Chow groups of smooth proper schemes over
global fields.

\begin{cor}\label{cor:global}
Let $F$ be a global field such that either $\ch(F)>0$ or $F$ is a
totally imaginary number field. Let $X$ be a smooth proper and
geometrically irreducible scheme over $F$ of dimension $d$.
Assume that $i\ge2$. Then the following hold.
\begin{enumerate}
\item If $2i-j\ge3$, then
$CH^{d+i}(X,j)=A^{d+i}(X,j)$ and this group is uniquely divisible.
\item If $2i-j=2$, there is a canonical short exact sequence
\[
0\to A^{d+i}(X,j)
\to CH^{d+i}(X,j)
\xrightarrow{f_*}CH^i(F,j)\to0,
\]
where $CH^i(F,j)$ is an infinite torsion group and
$A^{d+i}(X,j)$ is divisible. This sequence splits non-canonically.
\item If $2i-j=-2d$, then
$CH^{d+i}(X,j)=A^{d+i}(X,j)$ is torsion free.
\item If $2i-j<-2d$, then
$CH^{d+i}(X,j)=A^{d+i}(X,j)$ is uniquely divisible.
\end{enumerate}
\end{cor}

\begin{proof}
By the assumption on $F$, we have $\cd_l(F)=2$ for every prime
$l\neq\ch(F)$. If $\ch(F)=p>0$, then $[F:F^p]=p$.

\smallskip
\noindent
(1)\ By \autoref{lem:global}, we have $CH^i(F,j)=0$, and hence
$A^{d+i}(X,j)=CH^{d+i}(X,j)$. For every prime
$l\neq\ch(F)$, \autoref{thm:CH-div}\,(2) shows that this group is
uniquely $l$-divisible. If $\ch(F)=p>0$,
\autoref{prop:uniq-p-div} shows that it is uniquely $p$-divisible.
Therefore, $CH^{d+i}(X,j)=A^{d+i}(X,j)$ is uniquely divisible.

\smallskip
\noindent
(2)\ By \autoref{lem:global}, the group $CH^i(F,j)$ is an infinite torsion
group. For every prime $l\neq\ch(F)$,
\autoref{thm:CH-div}\,(1) shows that
\[
f_*\colon CH^{d+i}(X,j)[l^n]\to CH^i(F,j)[l^n]
\]
is surjective for every $n\ge1$.

If $\ch(F)=p>0$, then \autoref{prop:uniq-p-div}, applied to
$\Spec(F)$, shows that $CH^i(F,j)$ has no $p$-primary torsion.
Since $CH^i(F,j)$ is torsion, the preceding surjectivity on the
$l$-primary torsion subgroups shows that
\[
f_*\colon CH^{d+i}(X,j)\to CH^i(F,j)
\]
is surjective.

By \autoref{thm:CH-div}\,(2), the group $A^{d+i}(X,j)$ is
$l$-divisible for every prime $l\neq\ch(F)$. If $\ch(F)=p>0$, then
$A^{d+i}(X,j)$ is also $p$-divisible. Indeed, let
$a\in A^{d+i}(X,j)$. Since $CH^{d+i}(X,j)$ is uniquely
$p$-divisible by \autoref{prop:uniq-p-div}, there is a unique
$b\in CH^{d+i}(X,j)$ such that $a=pb$. Applying $f_*$ gives
$pf_*(b)=0$. Since $CH^i(F,j)$ is $p$-torsion free, we have
$f_*(b)=0$, and hence $b\in A^{d+i}(X,j)$. Thus
$A^{d+i}(X,j)$ is divisible.

We therefore have a short exact sequence
\[
0\to A^{d+i}(X,j)\to CH^{d+i}(X,j)
\xrightarrow{f_*}CH^i(F,j)\to0.
\]
Since a divisible abelian group is injective, this sequence splits
non-canonically. Hence
\[
CH^{d+i}(X,j)
\simeq
CH^i(F,j)\oplus A^{d+i}(X,j).
\]

\smallskip
\noindent
(3)\ By \autoref{lem:global}, we have $CH^i(F,j)=0$, so that
$A^{d+i}(X,j)=CH^{d+i}(X,j)$. By \autoref{prop:l-div}, this group is
$l$-torsion free for every prime $l\neq\ch(F)$. If $\ch(F)=p>0$,
then \autoref{prop:uniq-p-div} shows that it is also $p$-torsion
free. Therefore, $CH^{d+i}(X,j)=A^{d+i}(X,j)$ is torsion free.

\smallskip
\noindent
(4)\ Again, \autoref{lem:global} gives $CH^i(F,j)=0$, and hence
$A^{d+i}(X,j)=CH^{d+i}(X,j)$. For every prime
$l\neq\ch(F)$, \autoref{prop:l-div} shows that this group is uniquely
$l$-divisible. If $\ch(F)=p>0$, then
\autoref{prop:uniq-p-div} gives unique $p$-divisibility. Thus
$CH^{d+i}(X,j)=A^{d+i}(X,j)$ is uniquely divisible.
\end{proof}

\begin{cor}\label{cor:global-real}
Let $F$ be an arbitrary number field, and let $X$ be a smooth proper
and geometrically irreducible scheme over $F$ of dimension $d$.
Assume that $i\ge2$. Then the following hold.
\begin{enumerate}
\item If $2i-j\ge3$ or $2i-j<-2d$, there are canonical short exact
sequences
\[
0\to CH^{d+i}(X,j)[2]
\to CH^{d+i}(X,j)
\to D\to0
\]
and
\[
0\to A^{d+i}(X,j)[2]
\to A^{d+i}(X,j)
\to D_A\to0,
\]
where
\[
D:=CH^{d+i}(X,j)/CH^{d+i}(X,j)[2]
\quad\text{and}\quad
D_A:=A^{d+i}(X,j)/A^{d+i}(X,j)[2]
\]
are uniquely divisible.
\item If $2i-j=2$, the cokernel of
\[
f_*\colon CH^{d+i}(X,j)\to CH^i(F,j)
\]
is killed by $2$, and there is a canonical short exact sequence
\[
0\to A^{d+i}(X,j)[2]
\to A^{d+i}(X,j)
\to D_A\to0,
\]
where
\[
D_A:=A^{d+i}(X,j)/A^{d+i}(X,j)[2]
\]
is divisible.
\item If $2i-j=-2d$, the torsion subgroups of
$CH^{d+i}(X,j)$ and $A^{d+i}(X,j)$ are killed by $2$.
\end{enumerate}
\end{cor}

\begin{proof}
If $F$ is totally imaginary, the assertions follow from
\autoref{cor:global}. We may therefore assume that $F$ has a real
place. Put $F'=F(\sqrt{-1})$ and $\pi:X_{F'}\to X$. Then $F'$ is totally imaginary and
$[F':F]=2$. The pull-back and proper push-forward maps preserve the
kernels of the structure morphisms and satisfy
$\pi_*\pi^*=[F':F]$ 
on both $CH^{d+i}(X,j)$ and $A^{d+i}(X,j)$.

\smallskip
\noindent
(1)\ Assume that $2i-j\ge3$ or $2i-j<-2d$. By
\autoref{cor:global}, the corresponding groups over $F'$ are uniquely
divisible. For every odd prime $l$, \autoref{thm:CH-div} shows that
$A^{d+i}(X,j)$ is uniquely $l$-divisible. The same conclusion for
$CH^{d+i}(X,j)$ follows directly from \autoref{prop:l-div}, except
possibly when $2i-j=3$. In that case, \autoref{prop:l-div} gives
$l$-divisibility. If $x\in CH^{d+i}(X,j)[l]$, then
$\pi^*(x)=0$, since $CH^{d+i}(X_{F'},j)$ is torsion free. Hence
\[
2x=\pi_*\pi^*(x)=0.
\]
Since $l$ is odd and $lx=0$, we obtain $x=0$. Thus
$CH^{d+i}(X,j)$ is uniquely $l$-divisible also when $2i-j=3$.
It remains to consider multiplication by $2$ modulo the subgroups
killed by $2$.

Let $x\in CH^{d+i}(X,j)$. Since $\pi^*(x)$ is divisible by $4$, choose
$y\in CH^{d+i}(X_{F'},j)$ such that $\pi^*(x)=4y$. Then
\[
2\bigl(2\pi_*(y)-x\bigr)=0,
\]
so the class of $x$ in $CH^{d+i}(X,j)/CH^{d+i}(X,j)[2]$ is twice the class of
$\pi_*(y)$. Thus the quotient $CH^{d+i}(X,j)/CH^{d+i}(X,j)[2]$ is $2$-divisible. If the class of $x$ is
killed by $2$, then $4x=0$. Since the group over $F'$ is torsion
free, $\pi^*(x)=0$, and hence $2x=\pi_*\pi^*(x)=0$. Therefore
$CH^{d+i}(X,j)/CH^{d+i}(X,j)[2]$ is $2$-torsion free. The same argument applies to $A^{d+i}(X,j)$.
This proves (1).

\smallskip
\noindent
(2)\ Suppose that $2i-j=2$. Over $F'$, the push-forward is surjective by
\autoref{cor:global}. If $a\in CH^i(F,j)$, choose a lift of
$\res_{F'/F}(a)$ over $F'$. Pushing it forward to $F$ shows that
$2a$ lies in the image of $f_*$. Hence $\Coker(f_*)$ is killed by
$2$. Moreover, $A^{d+i}(X_{F'},j)$ is divisible. The same argument
as above, without the torsion-freeness step, shows that
$A^{d+i}(X,j)/A^{d+i}(X,j)[2]$ is $2$-divisible. It is also $l$-divisible for every
odd prime $l$ by \autoref{thm:CH-div}. Hence $A^{d+i}(X,j)/A^{d+i}(X,j)[2]$ is
divisible, proving (2).

\smallskip
\noindent
(3)\ Assume that $2i-j=-2d$. The groups over $F'$ are torsion
free by \autoref{cor:global}, while \autoref{prop:l-div} and
\autoref{thm:CH-div} show that $CH^{d+i}(X,j)$ and $A^{d+i}(X,j)$ have no odd-primary
torsion. If $x$ is a torsion element, then $\pi^*(x)=0$, and hence
$2x=\pi_*\pi^*(x)=0$. This proves (3).
\end{proof}

\begin{conj}[{Motivic Bass conjecture, \cite[Conjecture~37]{Kah05}}]
\label{conj:Bass}
Let $Y$ be a regular scheme of finite type over $\Z$. Then
$H^a(Y,\Z(b))$ is finitely generated for all integers $a$ and $b$.
\end{conj}

\begin{prop}\label{prop:global-Bass}
Assume the motivic Bass conjecture for regular schemes of finite type
over $\Z$. Let $F$ be a global field such that either $\ch(F)>0$ or
$F$ is a totally imaginary number field. Let $X$ be a smooth proper
and geometrically irreducible scheme over $F$ of dimension $d$.
If $i\ge2$ and $2i-j\ge3$, then $CH^{d+i}(X,j)=0$.
\end{prop}

\begin{proof}
For dimensional reasons, we may assume $j\ge 1$. 
We first show that $CH^{d+i}(X,j)$ is finitely generated. Choose a
nonempty regular one-dimensional open model $U$ of $F$. After
shrinking $U$ if necessary, choose a smooth proper morphism
$\mathcal X\to U$ whose generic fiber is $X$.

Since $U$ is regular and $\mathcal X$ is smooth over $U$, the scheme
$\mathcal X$ is regular and of finite type over $\Z$. Hence the
motivic Bass conjecture shows that
$CH^{d+i}(\mathcal X,j)=H^{2d+2i-j}(\mathcal X,\Z(d+i))$ is finitely
generated.

For each closed point $x\in U$, put
$\mathcal X_x=\mathcal X\times_U\kappa(x)$. For a nonempty open
subset $V\subset U$, put $\mathcal X_V=\mathcal X\times_U V$.
Passing to the direct limit of the localization sequences for the open
immersions $\mathcal X_V\hookrightarrow\mathcal X$, as $V$ ranges
over the nonempty open subsets of $U$, and using the continuity of
higher Chow groups, we obtain an exact sequence
\begin{equation}\label{ex:cX}
CH^{d+i}(\mathcal X,j)\to CH^{d+i}(X,j)\to
\bigoplus_{x\in U_0}CH^{d+i-1}(\mathcal X_x,j-1).
\end{equation}
Each $\mathcal X_x$ is smooth proper over the finite field
$\kappa(x)$. Since $2(i-1)-(j-1)=2i-j-1\ge2$,
\autoref{cor:fin2} shows that
$CH^{d+i-1}(\mathcal X_x,j-1)$ is uniquely divisible. On the other
hand, the motivic Bass conjecture shows that this group is finitely
generated. Therefore,
$CH^{d+i-1}(\mathcal X_x,j-1)=0$ for every $x\in U_0$.

It follows from \eqref{ex:cX} that
$CH^{d+i}(\mathcal X,j)\to CH^{d+i}(X,j)$ is surjective. Hence
$CH^{d+i}(X,j)$ is finitely generated. By
\autoref{cor:global}, the group $CH^{d+i}(X,j)$ is uniquely
divisible. A finitely generated uniquely divisible group is zero.
Thus $CH^{d+i}(X,j)=0$.
\end{proof}

\begin{cor}\label{cor:global-real-Bass}
Assume the motivic Bass conjecture for regular schemes of finite type
over $\Z$. Let $F$ be an arbitrary number field, and let $X$ be a
smooth proper and geometrically irreducible scheme over $F$ of
dimension $d$. If $i\ge2$ and $2i-j\ge3$, then
$CH^{d+i}(X,j)$ is killed by $2$.
\end{cor}

\begin{proof}
The proof of \autoref{prop:global-Bass} shows that
$CH^{d+i}(X,j)$ is finitely generated. By
\autoref{cor:global-real}\,(1), its quotient by the subgroup killed
by $2$ is uniquely divisible. A finitely generated uniquely divisible
group is zero. Hence $CH^{d+i}(X,j)$ is killed by $2$.
\end{proof}


\end{document}